\documentclass[a4wide, 12pt]{amsart}
\usepackage{amsmath, amsthm, amscd, amssymb}
\usepackage[normalem]{ulem}
\usepackage{tikz}
\usepackage{caption}
\usepackage{subcaption}
\usepackage{soul}
\numberwithin{equation}{section}
\theoremstyle{plain}
\newtheorem{theorem}{Theorem}[section]
\newtheorem{proposition}[theorem]{Proposition}
\newtheorem{lemma}[theorem]{Lemma}

\numberwithin{equation}{section}

\theoremstyle{definition}
\newtheorem{definition}[theorem]{Definition}
\newtheorem{remark}[theorem]{Remark}
\newtheorem{example}[theorem]{Example}

\DeclareMathOperator{\rank}{rank}

\def\cut{\mathrm{cut}}
\def\Z{\mathbb{Z}}
\def\Q{\mathbb{Q}}
\def\R{\mathbb{R}}
\def\C{\mathbb{C}}
\def\cF{\cF}
\def\e{\mathbf{e}}

\def\B{\mathcal{B}}

\def\cF{\mathcal{F}}
\def\cC{\mathcal{C}}
\def\cR{\mathcal{R}}
\def\St{\mathrm{St}\,}
\def\Lk{\mathrm{Lk}\,}
\def\PP{\mathrm{Poin}}

\def\blue#1{\textcolor{blue}{#1}}
\def\red#1{\textcolor{red}{#1}}

\tikzstyle{v}=[circle, draw, solid, fill=black!50, inner sep=0pt, minimum width=4pt]
\newcommand{\row}{\mathrm{Row}}

\begin{document}
\title{Pseudograph and its associated real toric manifold}
\author{Suyoung Choi}
\address{Department of mathematics, Ajou University,
Suwon 16499,
Republic of Korea}
\thanks{The first named author was supported by Basic Science Research Program through the National Research Foundation of Korea(NRF) funded by the Ministry of Science, ICT \& Future Planning(NRF-2012R1A1A2044990).}
\email{schoi@ajou.ac.kr}

\author{Boram Park}
\address{Department of mathematics, Ajou University,
Suwon 16499, Republic of Korea}\thanks{The second named author was supported by Basic Science Research Program through the National Research Foundation of Korea (NRF) funded by the Ministry of  Science, ICT \& Future Planning (NRF-2015R1C1A1A01053495)}
\email{borampark@ajou.ac.kr}

\author{Seonjeong Park}
\address{Osaka City University Advanced Mathematical Institute (OCAMI), 3-3-138 Sugimoto, Sumiyoshi-ku, Osaka-shi, 558-8585, Japan}
\email{seonjeong1124@gmail.com}



\date{\today}
\maketitle

\begin{abstract}
    Given a simple graph $G$, the graph associahedron $P_G$ is a convex polytope whose facets correspond to the connected induced subgraphs of $G$. Graph associahedra have been studied widely and are found in a broad range of subjects. Recently, S. Choi and H. Park computed the rational Betti numbers of the real toric variety corresponding to a graph associahedron under the canonical Delzant realization.
    In this paper, we focus on a pseudograph associahedron which was introduced by Carr, Devadoss and Forcey, and then discuss how to compute the Poincar\'{e} polynomial of the real toric variety corresponding to a pseudograph associahedron under the canonical Delzant realization.
\end{abstract}

\section{Introduction}\label{sec1}
    A \emph{toric variety} of complex dimension $n$ is a normal algebraic variety over $\C$ with an effective algebraic action of $(\C \setminus \{O\})^n$ having an open dense orbit. A compact smooth toric variety is called a \emph{toric manifold}, and the subset consisting of points with real coordinates of a toric manifold is called a \emph{real toric manifold}.
    One of the most important facts in toric geometry, the so called \emph{fundamental theorem of toric geometry}, is that there is a 1-1 correspondence between the class of toric varieties of complex dimension $n$ and the class of fans in $\R^n$. In particular, a toric manifold $X$ of complex dimension $n$ corresponds to a complete non-singular fan $\Sigma_X$ in $\R^n$. Furthermore, if $X$ is projective, then $\Sigma_X$ can be realized as the normal fan of some simple polytope of real dimension $n$.
    A simple polytope $P$ of dimension $n$ is called \emph{Delzant} if for each vertex $p \in P$, the outward normal vectors of the facets containing $p$ can be chosen to make up an integral basis for $\Z^n$. Note that the normal fan of a Delzant polytope is a complete non-singular fan and thus defines a projective toric manifold and a real toric manifold by the fundamental theorem as well.

    The (integral) Betti numbers of a toric manifold $X$ of complex dimension $n$ are easily computed. It is known by the theorem of Danilov\cite{Dan78}-Jurkiewicz\cite{Jur80} that the Betti numbers of $X$ vanish in odd degrees and the $2i$th Betti number of $X$ is equal to $h_i$, where $(h_0, h_1, \ldots, h_n)$ is the $h$-vector of $\Sigma_X$.
    Unlike toric manifolds, however, only little is known about the topology of real toric manifolds.
    In~\cite{ST2012} and \cite{Tre2012}, Suciu and Trevisan have found a formula for the rational cohomology groups of  a real toric manifold, see also \cite{CP2}. Let $P$ be a Delzant polytope of dimension~$n$. Let $\cF=\{F_1,\ldots,F_m\}$ be the set of facets of $P$. Then, the outward normal vectors of $P$ can be understood as a function $\phi$ from $\cF$ to $\Z^n$, and the composition map $\lambda \colon \cF \stackrel{\phi}{\to} \Z^n \stackrel{\text{mod $2$}}{\longrightarrow} \Z_2^n$ is called the (mod $2$) \emph{characteristic function} over $P$. Note that $\lambda$ can be represented by a $\Z_2$-matrix $\Lambda$ of size $n \times m$ as
    $$
    \Lambda = \begin{pmatrix}
      \lambda(F_1) & \cdots & \lambda(F_m)
    \end{pmatrix},
    $$ where the $i$th column of $\Lambda$ is $\lambda(F_i) \in \Z_2^n$.
    For $\omega \in \Z_2^m$, we define $P_\omega$ to be the union of facets $F_j$ such that the $j$th entry of $\omega$ is nonzero.
    Denote by $M_\lambda(P)$ the real toric manifold corresponding to $P$ and $\lambda$. Then the $i$th rational Betti number of $M_\lambda(P)$ is given by
    \begin{equation}
        \beta^i (M_\lambda(P)) = \sum_{\omega\in \row(\lambda)} \rank_{\Q} \widetilde H^{i-1}(P_\omega;\Q), \label{eq:formula}
    \end{equation}
    where $\row(\lambda)$ is the space of $m$-dimensional $\Z_2$-vectors spanned by the rows of
    $\Lambda$ associated with $\lambda$.
    S. Choi and H. Park \cite{CP2} reproved the formula \eqref{eq:formula} by using different methods, and showed that it holds even for the cohomology group of arbitrary coefficient ring in which $2$ is a unit.
    Furthermore, it is known in \cite{CKT} that the suspension of $M$ localized at an odd prime $p$ (or $p=0$, which is known as the rationalization) can be decomposed as
    \begin{equation}\label{eq:sus}
       \Sigma M \simeq_p \Sigma \bigvee_{\omega \in \row(\lambda)} \Sigma P_\omega.
    \end{equation}

    Recently, the rational Betti numbers of some interesting family of real toric manifolds, arising from simple graphs, have been computed by using \eqref{eq:formula} in \cite{CP}.
    Let $G$ be a simple graph with node set $[n+1]:=\{1,\ldots,n+1\}$ and $\B(G)$ the \emph{graphical building set} which is the collection of subsets of $[n+1]$ all of whose elements are obtained from connected induced subgraphs of $G$.
    For $I\subset [n+1]$, let $\Delta_I$ be the simplex given by the convex hull of points $\e_i$, $i\in I$, where $\e_i$ is the $i$th standard basis vector. Then define the \emph{graph associahedron} $P_G$ as the Minkowski sum of simplices
    \begin{equation}
      P_G = \sum_{I\in\B(G)} \Delta_I. \label{eq:Minkowski_sum}
    \end{equation}
    One can see that $G$ is connected if and only if $P_G$ is $n$-dimensional. We note that, for every connected graph $G$ with $n+1$ nodes, there is a canonical way to realize $P_G$ as a Delzant polytope from $\Delta^n$ in $\R^n$ by truncating its faces corresponding to elements of $\B(G) \setminus \{[n+1]\}$ (refer \cite{Zel2006} or \cite[Proposition 7.10]{P05}).
    We note that to each facet of~$P_G$ there is an assigned element $I$ of $\B(G) \setminus \{[n+1]\}$. In this case, by choosing an appropriate basis of $\Z_2^n$, the characteristic function $\lambda_G$ is
    $$
    \lambda_G (I) = \left\{
                      \begin{array}{ll}
                        \sum_{i \in I} \e_i, & \hbox{if $n+1 \not\in I$;} \\
                        \sum_{i \not\in I} \e_i, & \hbox{if $n+1 \in I$}.
                      \end{array}
                    \right.
    $$
    We denote by $M_G$ the real toric manifold $M_{\lambda_G}(P_G)$ corresponding to $P_G$ under the canonical Delzant realization.

    In \cite{CP}, the rational Betti numbers of $M_G$ are computed in terms of the $a$-numbers of $G$. The authors noted that, for each simple graph $G$ with an even number of nodes, there is the unique element $\omega_0(G) \in \row(\lambda_G)$ such that the component of $\omega_0(G)$ corresponding to every singleton in $\B(G)$ is nonzero. The remarkable fact discovered in \cite{CP} is that for $\omega \in \row(\lambda_G)$, $(P_G)_\omega$ is homotopy equivalent to $(P_H)_{\omega_0(H)}$ for some subgraph $H$ of $G$ induced by an even number of nodes. Furthermore, $(P_H)_{\omega_0(H)}$ is homotopy equivalent to a wedge of spheres of the same dimension. Therefore, they concluded that $(P_G)_\omega$ is homotopy equivalent to a wedge of spheres of the same dimension for each $\omega \in \row(\lambda_G)$.

    In this paper, we discuss the pseudographical analogue of the above phenomena. A \emph{pseudograph} $G$ is a graph in which both loops and multiple edges are permitted. The notion of a pseudograph associahedron is originally introduced in \cite{CDF2011} as a generalization of a graph associahedron. Indeed, it coincides with the graph associahedron $P_G$ if $G$ is a connected simple graph. However, if $G$ is disconnected, then it does not coincide with $P_G$ since the dimension of the corresponding pseudograph associahedron  in \cite{CDF2011} is always greater than the dimension of the graph associahedron. So in this paper, we introduce a slightly modified definition of a pseudograph associahedron, see Section~\ref{sec2}, which totally agrees with the notion of a graph associahedron when $G$ is simple.
    Then, we will denote again by $P_G$ the (our modified) pseudograph associahedron of a pseudograph~$G$.
    We also remark that if $G$ does not have a loop, then there is the canonical Delzant realization of $P_G$, and, hence, it provides a real toric manifold ${M_G}$.

    For a topological space $X$, if $X$ has the homotopy type of a finite CW-complex, then we let $\PP_X(t)$ and $\widetilde{\PP}_X(t)$ be the polynomials defined by
    \[
        \PP_X(t)=\sum_{i=0}^{\infty} \beta^i(X) t^i \quad { \text{ and } } \quad \widetilde{\PP}_X(t)=\sum_{i=0}^{\infty} \tilde{\beta}^i(X) t^i ,
    \]
    where $\beta^i(X)$ is the $i$th rational Betti number of $X$ and $\tilde{\beta}^i(X)$ is the $i$th reduced rational Betti number of $X$. The polynomial $\PP_X(t)$ is called the \emph{Poincar\'{e} polynomial} of $X$.
     Our main result is about computing the Poincar\'{e} polynomial of $M_G$.

    A pseudograph $H$ is an \emph{induced subgraph} of $G$ if $H$ is a subgraph that includes all edges between every pair of nodes in $H$ if such edges of $G$ exist. A pseudograph $G'$ is a \emph{partial underlying pseudograph} of $G$ if $G'$ can be obtained from $G$ by replacing some bundles with simple edges.
    We denote by $H\lessdot G$ if $H$ is a partial underlying pseudograph of an induced subgraph of $G$.
    The main result of this paper may be stated as follows, and all definitions and notations related to Theorem~\ref{thm:main} are explained  in Section~\ref{sec3}, with more observations and examples. Throughout this paper, for simplicity, we also denote by $K$ the topological realization of a simplicial complex $K$ if there is no danger of confusion.
    \begin{theorem} \label{thm:main}
        For any pseudograph $G$, we have
        \[\PP_{M_G}(t)=1+t\sum_{H\lessdot G}\sum_{C\subset\cC_{H}\atop\text{admissible to $H$}}\widetilde{\PP}_{K_{C,H}^{\mathrm{odd}}}(t),\]
        where $\cC_H$ is the set of all nodes and all multiple edges of $H$.
    \end{theorem}
    Roughly, for $C\subset\cC_{H}$, $C$ is said to be admissible to $H$ if $C$ has an even number of nodes and an even number of multiple edges of $H$ satisfying certain conditions, and $K_{C,H}^{\mathrm{odd}}$ is a subcomplex of the dual of $ \partial P_H$ satisfying certain conditions related to $C$.

    In our formula, {\tiny$\displaystyle \sum_{C\subset\cC_{H}\atop {\text{admissible to } H}}\widetilde{\PP}_{K_{C,H}^{\mathrm{odd}}}(t)$}
    is completely determined by $H\lessdot G$, 
    and so we define it by the \emph{$a$-polynomial} $a_H(t)$ of $H$.
    Interestingly,  our main theorem says that for the computation of  $\PP_{M_G}(t)$, it is sufficient to consider some subgraphs of $G$, instead of considering $P_{\omega}$ for all $\omega\in \row(\lambda_G)$.
    One of the main ideas to prove Theorem~\ref{thm:main} is to show that for $\omega \in \row(\lambda_G)$, $(P_G)_\omega$ is homotopy equivalent to $K_{C,H}^{\mathrm{odd}}$ for some collection $C$ and some $H\lessdot G$, 
    see Proposition~\ref{prop:LC}.
    Hence, not only this result reduces the computation of $\PP_{M_G}(t)$ remarkably, but enhances the understanding of the topology of suspended $M_G$ as \eqref{eq:sus}, since the dimension of $K_{C,H}^{\mathrm{odd}}$ is much less than that of $(P_G)_\omega$.

    Our result is a generalization of \cite{CP} in a sense that the $a$-polynomial of a simple graph $H$ corresponds to the $a$-number of $H$. More precisely, if a graph $G$ is simple, then an induced subgraph $H$ of $G$ having an admissible collection must have an even number of nodes and its node set is a unique admissible collection of $H$, and therefore the $a$-polynomial $a_H(t)$ is a monomial of degree $\frac{|V(H)|}{2}$ and the coefficient of the monomial is exactly equal to the $a$-number of $H$.

    This paper is organized as follows:
    in Section~\ref{sec2}, we review the slightly modified definition of a pseudograph associahedron, and show that pseudograph associahedra are Delzant. In Section~\ref{sec3}, we define the $a$-polynomial of a pseudograph, and prepare the notions necessary to state the main theorem.
    In Section~\ref{sec4}, we show that $(P_G)_\omega$ is homotopy equivalent to $K_{C,H}^{\mathrm{odd}}$ and prove the main theorem.
    We conclude in Section~\ref{sec6:remark} with some remarks.

\section{Pseudograph associahedron and its associated real toric manifold} \label{sec2}
    In this section, we briefly review the construction of the pseudograph associahedron $P_G$ for a pseudograph $G$ based on \cite{CDF2011}. However, we modify some notions in order to improve the explanation, but the essential ideas are the same.

    A \emph{pseudograph} $G$ is an ordered pair $G:=(V, E)$, where $V$ is a  set of \emph{nodes}\footnote{In this paper, we use `node' for a graph or a pseudograph, and we use `vertex' for a 0-dimensional simplex.} and $E$ is a multiset of unordered pairs of nodes, called \emph{edges}. If the endpoints of an edge $e$ are the same then $e$ is called a \emph{loop}. A pseudograph is said to be \emph{finite} if both $V$ and $E$ are finite sets. Throughout this paper, we only consider a finite pseudograph.
    An edge $e\in E$ is said to be \emph{multiple} if there exists an edge $e'$ ($\neq e$) in $E$ such that $e$ and $e'$ have the same pair of endpoints.
    An edge is called a \emph{simple edge} if it is not a multiple edge.
    A \emph{bundle} is a maximal set of multiple edges which have the same pair of endpoints. For example, the pseudograph $G$ in Figure~\ref{fig:Ex1} on page~\pageref{fig:Ex1} has five multiple edges and two bundles.

    The \emph{underlying simple graph}
    of $G$ is created by deleting all loops and replacing each bundle with a simple edge.
        A pseudograph $G$ is said to be \emph{connected} if its underlying simple graph is connected.
    A pseudograph $H$ is an \emph{induced} (respectively, \emph{semi-induced}) subgraph of $G$ if $H$ is a subgraph that includes all edges (respectively, at least one edge) between every pair of nodes in $H$ if such edges exist in $G$.
    Note that for a simple graph, semi-induced subgraphs and induced subgraphs are the same concept. For example, in Figure~\ref{fig:Ex1}, $I_1$, $I_4$, and $I_5$ are induced subgraphs and $I_3$ is a semi-induced subgraph of $G$.

    \begin{definition}
        Let $G$ be a (not necessarily connected) pseudograph. A \emph{tube} is a proper connected {semi-induced} subgraph $I$ of some connected component of $G$.\footnote{This is different from the definition of a tube in \cite{CD2006} by Carr and Devados. They defined a tube by a proper connected subgraph that includes at least one edge between every pair of nodes in $I$ if such edges  exist, and so a connected component itself can be a tube while our definition does not allow it.}
        A tube is  said to be  \emph{full} if it is an induced subgraph of $G$. Two tubes  \emph{meet by inclusion} if one properly contains the other, and they  \emph{meet by separation} if they are disjoint and cannot be connected by an edge of $G$.
        Two tubes are \emph{compatible} if they meet by inclusion or separation.
        A \emph{tubing} of $G$ is a set of pairwise compatible tubes.
    \end{definition}

    From now on, we will consider only pseudographs without loops. We give labels to the nodes and the multiple edges of a pseudograph.
    In addition, when we consider a subgraph $H$ of a pseudograph $G$, the labels of $H$ are inherited from the labels of $G$.
    Thus, if a pseudograph $H$ is considered as a subgraph of a pseudograph $G$, then $H$ might have a labelled simple edge, which is not in a bundle of $H$ (actually, it is in a bundle of $G$).
    Here are examples for tubes and tubings.

    \begin{example}\label{example:C}
        Let $G$ be a pseudograph labelled as Figure~\ref{fig:Ex1}. Then $G$ has two bundles $\{a,b\}$ and $\{c,d,e\}$.
        The subgraphs $I_1$, $I_3$, $I_4$, and $I_5$ are tubes of $G$, but $I_2$ is not a tube. In particular,
        $I_1$, $I_4$, and $I_5$ are full tubes. Note that $\{I_1,I_4,I_5\}$ is a tubing since $I_1$ and $I_4$ meet by inclusion, $I_4$ and $I_5$ meet by inclusion, and $I_1$ and $I_5$ meet by separation. But, $\{I_1, I_3, I_4\}$ is not a tubing since $I_3$ and $I_4$ are not compatible.
        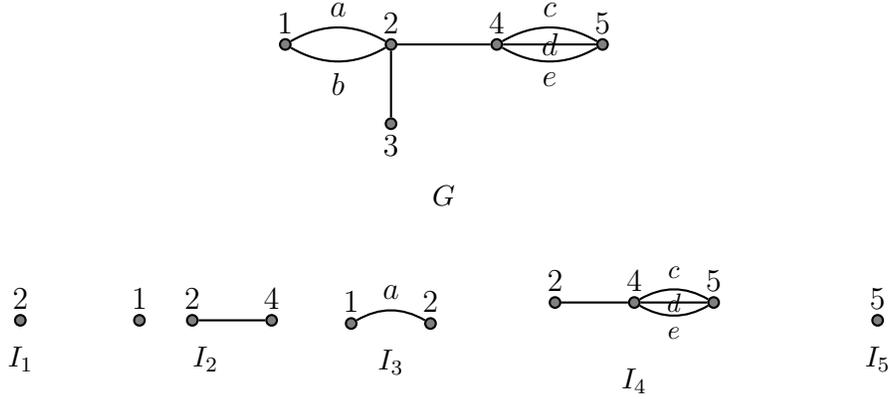
\begin{figure}[h!]
        \begin{center}
        \begin{subfigure}[c]{0.7\textwidth}
        \centering
        \begin{tikzpicture}[thick,scale=0.5]
            \node [v] (v1) {} (v1) node[above]{1};
            \node [v] (v2) [right of=v1, node distance=40pt]{} (v2) node[above]{2};
            \node [v] (v3) [below of=v2, node distance=30pt]{} (v3) node[below]{3};
            \node [v] (v4) [right of=v2, node distance=40pt]{} (v4) node[above]{4};
            \node [v] (v5) [right of=v4, node distance=40pt]{} (v5) node[above]{5};
            \path
            (v1) edge[bend left] node[above] {$a$} (v2)
                 edge[bend right] node[below] {$b$} (v2)
            (v2) edge (v3)
                 edge (v4)
            (v4) edge[bend left] node[above] {$c$} (v5)
                 edge node {$d$} (v5)
                 edge[bend right] node[below] {$e$} (v5);
        \end{tikzpicture}
        \caption*{$G$}
        \end{subfigure}
        \vspace{.5cm}
        \\
        \begin{subfigure}[c]{.1\textwidth}
        \centering
        \begin{tikzpicture}[thick,scale=0.4]
            \node [v] (v1) {} (v1) node[above]{2};
        \end{tikzpicture}
        \caption*{$I_1$}
        \end{subfigure}
        \begin{subfigure}[c]{.2\textwidth}
        \centering
        \begin{tikzpicture}[thick,scale=0.4]
            \node [v] (v1) {} (v1) node[above]{1};
            \node [v] (v2) [right of=v1, node distance=20pt]{} (v2) node[above]{2};
            \node [v] (v4) [right of=v2, node distance=30pt]{} (v4) node[above]{4};
            \path(v2) edge  (v4);
        \end{tikzpicture}
        \caption*{$I_2$}
        \end{subfigure}
        \begin{subfigure}[c]{.1\textwidth}
        \centering
        \begin{tikzpicture}[thick,scale=0.4]
            \node [v] (v1) {} (v1) node[above]{1};
            \node [v] (v2) [right of=v1, node distance=30pt]{} (v2) node[above]{2};
            \path
            (v1) edge[bend left] node[above] {$a$} (v2);
        \end{tikzpicture}
        \caption*{$I_3$}
        \end{subfigure}
        \begin{subfigure}[c]{.3\textwidth}
        \centering
        \begin{tikzpicture}[thick,scale=0.4]
            \node [v] (v2) [right of=v1, node distance=30pt]{} (v2) node[above]{2};
            \node [v] (v4) [right of=v2, node distance=30pt]{} (v4) node[above]{4};
            \node [v] (v5) [right of=v4, node distance=30pt]{} (v5) node[above]{5};
            \path
            (v2) edge (v4)
            (v4) edge[bend left] node[above] {\footnotesize $c$} (v5)
                 edge node {\footnotesize $d$} (v5)
                 edge[bend right] node[below] {\footnotesize $e$} (v5)
           ;
        \end{tikzpicture}
        \caption*{$I_4$}
        \end{subfigure}
        \begin{subfigure}[c]{.1\textwidth}
        \centering
        \begin{tikzpicture}[thick,scale=0.4]
            \node [v] (v5) [right of=v4, node distance=30pt]{} (v5) node[above]{5};
        \end{tikzpicture}
        \caption*{$I_5$}
        \end{subfigure}
        \caption{A pseudograph $G$ and its subgraphs}\label{fig:Ex1}
        \end{center}
        \end{figure}

    \end{example}

    For a pseudograph $G$, a subgraph $I$ of $G$ will be denoted by the set of nodes of $I$ and edges of $I$ in a bundle of $G$.
    For instance, for the five subgraphs of $G$ in Figure~\ref{fig:Ex1},
    \[I_1=\{2\},\quad I_2=\{1,2,4\},\quad I_3=\{1,2,a\},\quad I_4=\{2, 4, 5, c, d, e \},\quad \text{and} \quad I_5=\{5\}.\]
    It should be noted that for a semi-induced subgraph $I$, this set expression makes sense because $I$ is the subgraph of $G$ induced by the corresponding set.
    In the same sense, for a subgraph $I$ of $G$, we denote by $\alpha \in I$ if $\alpha$ is a node of $I$ or a multiple edge of $G$.
    From now on, for simplicity, we omit the braces and commas, and we always denote it in a way that the nodes proceed to the multiple edges, and the nodes are arranged in increasing order, like as $I_1=2$, $I_3=12a$, $I_4=245cde$, and $I_5=5$.

    Before we define a polyhedron associated with a pseudograph $G$, we introduce a specific labelling $L_I$ corresponding to a tube $I$ of $G$.
    For each tube $I$ of $G$, $L_I$ is defined to be the minimal collection of nodes and edges of $G$ such that
    \begin{itemize}
        \item[-] $L_I$ contains $I$ as a set, and
        \item[-] if a bundle $B$ of $G$ satisfies that $B\cap I=\emptyset$, then $B\subset L_I$.
    \end{itemize}
    For simplicity, we  also  omit the braces and commas when we denote $L_I$.
    For instance, the tubes $I_1$, $I_3$, $I_4$, and $I_5$ in Figure~\ref{fig:Ex1} have the associated labels
    \[
        L_{I_1} = 2abcde,\quad
        L_{I_3} = 12acde,\quad
        L_{I_4} = 245abcde,\quad and \quad
        L_{I_5}=5abcde.
    \]

    Now we are ready to define a pseudograph associahedron corresponding to a finite pseudograph without loops.\footnote{In \cite{CDF2011}, a pseudograph associahedron is defined for any pseudograph even if it has a loop. However, in this case, the pseudograph associahedron is not bounded. In this paper, since we are interested in only Delzant polytopes that are bounded, we shall deal with only pseudographs having no loops.} Let $G$ be a connected pseudograph with $n+1$ nodes.  Let $B_1,\ldots,B_k$ be the bundles of $G$, and $b_i+1$ the number of edges in $B_i$ for $i=1,\ldots, k$.
    Define $\Delta_G$ as the product
    $$\Delta_G:=\Delta^{n}\times\prod_{i=1}^k \Delta^{b_i}$$
    of simplices endowed with the following labels on its faces:
    \begin{enumerate}
        \item each facet of the simplex $\Delta^{n}$ is labelled with a particular node of $G$, and each face of $\Delta^{n}$ corresponds to a proper subset of nodes of $G$, defined by the intersection of the facets associated with those nodes;
        \item each vertex of the simplex $\Delta^{b_i}$ is labelled with a particular edge in $B_i$, and each face of $\Delta^{b_i}$ corresponds to a subset of $B_i$ defined by the vertices spanning the face; and
        \item these labels naturally induce a labelling  on $\Delta_G$.
    \end{enumerate}

    We note that a full tube of $G$ is determined by the element of the graphical building set $\B(G^s)$, where $G^s$ is the underlying simple graph of $G$.
    Hence, one can see that, for a connected pseudograph $G$, truncating the faces of $\Delta_G$ labelled by $L_I$ for full tubes $I$, in increasing order of dimension, constructs $$\Delta_G^\ast:=P_{G^s}\times\prod_{i=1}^{k} \Delta^{b_i}. $$
    As each face $f$ of $\Delta^*_G$ corresponding to a full tube is truncated, those subfaces of $f$ that correspond to tubes but have not yet been truncated are removed. It is natural, however, to assign these defunct tubes to the combinatorial images of their original subfaces.

    \begin{definition}[Modifying the definition in \cite{CDF2011}]\label{def:pseudograph associahedron}
        Let $G$ be a pseudograph. If $G$ is connected, then truncating the remaining faces of $\Delta_G^\ast$ labelled with $L_I$ for tubes $I$, in increasing order of dimension, results in the pseudograph associahedron $P_G$. If $G$ is a pseudograph with connected components $G^1,\ldots,G^q$, then the pseudograph associahedron is $P_G= P_{G^1}\times\cdots\times P_{G^q}$.\footnote{
    We note that, for a connected simple graph $G$, we have $P_G = Q_{G}$, where $P_{G}$ is the graph associahedron corresponding to $G$ as in \eqref{eq:Minkowski_sum}, and $Q_G$ is the polytope defined by the original definition of pseudograph associahedron in~\cite{CDF2011}. However, if $G$ is not connected, that is, $G$ has connected components $G^1, \ldots, G^q$, then $Q_G$ is defined to be $Q_G= P_{G^1}\times\cdots\times P_{G^q} \times \Delta^{q-1}$ while $P_{G} = P_{G^1}\times\cdots\times P_{G^q}$ by \eqref{eq:Minkowski_sum}. In order to avoid this confusion, we shall use a slightly modified definition of pseudograph associahedron as in Definition~\ref{def:pseudograph associahedron}.}
    \end{definition}

    \begin{example}\label{ex:associahedron}
        Let  $G$ be a pseudograph with three nodes and one bundle of size two labelled as in Figure~\ref{fig:K3}.
        Then the pseudograph $G$ has nine tubes, and five of them are full. The following table shows all the tubes $I$ of $G$ and their labellings $L_I$.

        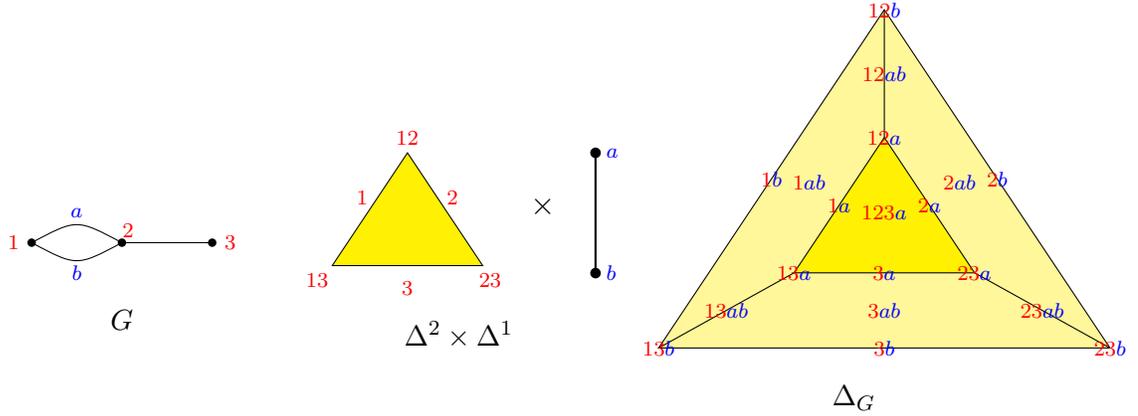
\begin{figure}[h]
        \begin{center}
        \begin{subfigure}[b]{0.28\textwidth}
        \centering
        \begin{tikzpicture}[scale=0.8]
    	\fill (3,0) circle(2pt);
    	\fill (0,0) circle(2pt);
    	\fill (1.5,0) circle(2pt);
        \draw (0,0)..controls (0.75,-0.4)..(1.5,0);	
        \draw (3,0)--(1.5,0);
    	\draw (0,0)..controls (0.75,0.4)..(1.5,0);
    	\draw (3.3,0) node{\red{\tiny{$3$}}};
    	\draw (-0.3,0) node{\red{\tiny{$1$}}};
    	\draw (1.6,0.2) node{\red{\tiny{$2$}}};
    	\draw (0.75,0.5) node{\blue{\tiny{$a$}}};
    	\draw (0.75,-0.5) node{\blue{\tiny{$b$}}};
    \end{tikzpicture}

    	\caption*{$G$}
        \vspace{1cm}
        \end{subfigure}
        \begin{subfigure}[b]{0.28\textwidth}
        \begin{tikzpicture}
    		\filldraw[fill=yellow] (0,0)--(2,0)--(1,3/2)--cycle;
    		\draw (-0.2,-0.2) node{\red{\tiny $13$}};
    		\draw (1,1.7) node{\red{\tiny $12$}};
    		\draw (2.1,-0.2) node{\red{\tiny $23$}};
    		\draw (1,-0.3) node{\red{\tiny $3$}};
    		\draw (1.6,0.9) node{\red{\tiny $2$}};
    		\draw (0.4,0.9) node{\red{\tiny $1$}};
            \draw (2.8,0.5)[above] node{{$\times$}};
    		\fill (3.5,3/2) circle(2pt);
    		\fill (3.5,-0.1) circle(2pt);
            \draw[thick] (3.5,3/2)--(3.5,-0.1);
    		\draw (3.5,3/2)[right] node{\blue{\tiny $a$}}
                (3.5,-0.1)[right] node{\blue{\tiny $b$}};
    	\end{tikzpicture}
    	\caption*{$\Delta^{2}\times\Delta^{1}$}
        \vspace{0.8cm}
    	\end{subfigure}
    	\begin{subfigure}[b]{0.38\textwidth}
    	\begin{tikzpicture}
    		\filldraw[fill=yellow!50] (0,0)--(6,0)--(3,4.5)--cycle;
    		\filldraw[fill=yellow] (1.8,1)--(4.2,1)--(3,2.8)--cycle;
            \draw (0,0) -- (1.8,1);
            \draw (6,0)--(4.2,1);
            \draw (3,4.5)--(3,2.8);
    		\draw (0,0) node  {\tiny \red{$13$}\blue{$b$}};
            \draw (6,0) node {\tiny \red{$23$}\blue{$b$}};
            \draw (3,4.5) node {\tiny \red{$12$}\blue{$b$}};
            \draw (1.8,1) node {\tiny \red{$13$}\blue{$a$}};
            \draw (4.2,1) node {\tiny \red{$23$}\blue{$a$}};
            \draw (3,2.8) node {\tiny  \red{$12$}\blue{$a$}};
            \draw (3,1.8) node {\tiny  \red{$123$}\blue{$a$}}
                (3,0.5) node {\tiny  \red{$3$}\blue{$ab$}}
                (2,2.2) node {\tiny  \red{$1$}\blue{$ab$}}
                (4,2.2) node {\tiny  \red{$2$}\blue{$ab$}};
            \draw (3,0) node {\tiny  \red{$3$}\blue{$b$}}
                (3,1) node {\tiny  \red{$3$}\blue{$a$}}
                (2.4,1.9) node {\tiny  \red{$1$}\blue{$a$}}
                (3.6,1.9) node {\tiny  \red{$2$}\blue{$a$}}
                (1.5,2.25) node {\tiny  \red{$1$}\blue{$b$}}
                (4.5,2.25) node {\tiny  \red{$2$}\blue{$b$}}
                (0.9,0.5) node {\tiny  \red{$13$}\blue{$ab$}}
                (5.1,0.5) node {\tiny  \red{$23$}\blue{$ab$}}
                (3,3.65) node {\tiny  \red{$12$}\blue{$ab$}};
    	\end{tikzpicture}
    	\caption*{$\Delta_G$}
    	\end{subfigure}
        \end{center}
        \caption{An example of $\Delta_G$ whose facets are labelled with $L_I$}\label{fig:K3}
    	\end{figure}

        \begin{center}
        \begin{tabular}{||c|c|c||c|c|c||c|c|c||}
              \hline
              $I$ & $L_I$ &  & I & $L_I$ & & I & $L_I$ & \\ \hline
              $1$ & $1ab$ & full & $23$ & $23ab$ & full & $12a$ & $12a$ & not full \\
              $2$ & $2ab$ & full & $12ab$ & $12ab$ & full & $12b$ & $12b$ & not full \\
              $3$ & $3ab$ & full &  &  & & $123a$ & $123a$ & not full\\
              & & & & && $123b$ & $123b$ & not full\\
              \hline
        \end{tabular}
        \end{center}

    First, we truncate the faces of~$\Delta_G$ corresponding to full tubes. Then we get the first polytope in Figure~\ref{fig:pseudograph associahedron} and the defunct tubes $12a$ and $12b$ correspond to some edges of the new facet obtained from the truncation of the face corresponding to $12ab$. Now we truncate the faces of $\Delta_{G}^\ast$ whose labels correspond to the remaining tubes. Then we obtain the pseudograph associahedron with nine facets labelled by $L_I$.
    See Figure~\ref{fig:pseudograph associahedron}.
        \begin{figure}[h]
        \begin{center}
    	\begin{subfigure}[b]{0.38\textwidth}
    	\begin{tikzpicture}
    		\filldraw[fill=yellow!50] (0,0)--(5.2,0)--(5.6,0.6)--(3.4,3.9)--(2.6,3.9)--cycle;
    		\filldraw[fill=yellow] (1.8,1)--(3.8,1)--(4,1.3)--(3.4,2.2)--(2.6,2.2)--cycle;
            \draw (0,0) -- (1.8,1);
            \draw (5.2,0)--(3.8,1);
            \draw (5.6,0.6)--(4,1.3);
            \draw (3.4,3.9)--(3.4,2.2);
            \draw (2.6,3.9)--(2.6,2.2);
            \draw (3,0.5) node {\tiny  \red{$3$}\blue{$ab$}}
                (2,2.2) node {\tiny  \red{$1$}\blue{$ab$}}
                (4,2.2) node {\tiny  \red{$2$}\blue{$ab$}}
                (5.1,0.5) node {\tiny  \red{$23$}\blue{$ab$}}
                (3,2.2) node {\tiny \red{$12$}\blue{$a$}}
                (3,3.9) node{\tiny \red{$12$}\blue{$b$}}
                (3,1.6) node {\tiny  \red{$123$}\blue{$a$}}
                (3,3) node {\tiny  \red{$12$}\blue{$ab$}};
    	\end{tikzpicture}
    	\caption*{$\Delta_G^\ast$}
    	\end{subfigure}
    	\begin{subfigure}[b]{0.38\textwidth}
    	\begin{tikzpicture}
    		\filldraw[fill=yellow!50] (0,0)--(5.2,0)--(5.6,0.6)--(3.6,3.6)--(2.4,3.6)--cycle;
    		\filldraw[fill=yellow] (1.8,1)--(3.8,1)--(4,1.3)--(3.5,2.05)--(2.5,2.05)--cycle;
            \filldraw[fill=yellow!80]
            (3.6,3.6)--(2.4,3.6)--(2.6,3.3)--(3.4,3.3)--cycle;
            \filldraw[fill=yellow!80]
            (3.5,2.05)--(2.5,2.05)--(2.6,2.45)--(3.4,2.45)--cycle;
            \draw (0,0) -- (1.8,1);
             \draw (5.2,0)--(3.8,1);
            \draw (5.6,0.6)--(4,1.3);
            \draw (3.4,3.3)--(3.4,2.45);
            \draw (2.6,3.3)--(2.6,2.45);
            \draw (3,3.45) node {\tiny \red{$12$}\blue{$b$}};
            \draw (3,2.25) node {\tiny  \red{$12$}\blue{$a$}};
            \draw (3,1.6) node {\tiny  \red{$123$}\blue{$a$}}
                (3,0.5) node {\tiny  \red{$3$}\blue{$ab$}}
                (2,2.2) node {\tiny  \red{$1$}\blue{$ab$}}
                (4,2.2) node {\tiny  \red{$2$}\blue{$ab$}}
                (5.1,0.5) node {\tiny  \red{$23$}\blue{$ab$}}
                (3,3) node {\tiny  \red{$12$}\blue{$ab$}};
    	\end{tikzpicture}
    	\caption*{$P_G$}
    	\end{subfigure}
    	\begin{subfigure}[b]{0.2\textwidth}
    	\begin{tikzpicture}[scale=.31]
    		\pgfsetfillopacity{0.8}
    		\draw (-1,0)--(3,3);
    		\draw (3,3)--(10,3);
    		\draw (3,3)--(5.25,7.75);
    		\draw (5,8.2)--(5.25,7.75)--(8.75,7.75);
    		\filldraw[fill=yellow] (-1,0)--(6,0)--(7,1)--(4.8,5)--(1.2,5)--cycle;
    		\filldraw[fill=yellow] (1.2,5)--(4.8,5)--(5,6.8)--(3,6.8)--cycle;
    		\filldraw[fill=yellow] (7,1)--(6,0)--(10,3)--(11,4)--cycle;
    		\filldraw[fill=yellow!50] (3,6.8)--(5,8.2)--(7,8.2)--(5,6.8)--cycle;
    		\filldraw[fill=yellow!50] (7,8.2)--(8.75,7.75)--(11,4)--(7,1)--(4.8,5)--(5,6.8)--cycle;
    	\end{tikzpicture}
    	\caption*{$P_G$}
    	\end{subfigure}
        \end{center}\caption{The pseudograph associahedron $P_G$}\label{fig:pseudograph associahedron}
        \end{figure}
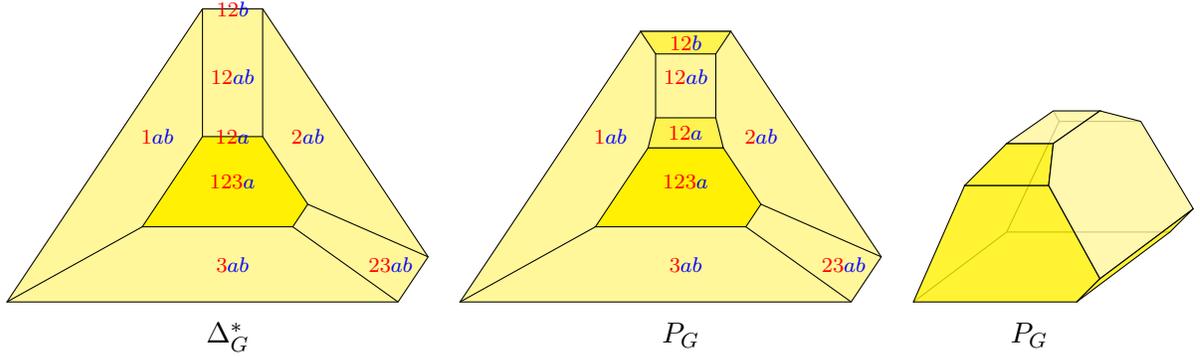
    \end{example}

    For simplicity, we will use the same notation for tubes and facets, that is, we denote by $I$ the facet labelled by $L_I$.
    The face poset of $P_G$ is isomorphic to the set of tubings of $G$, ordered under reverse subset containment.
    Two facets \emph{meet by inclusion}  (respectively, \emph{meet by separation})  if their corresponding tubes meet by inclusion (respectively, meet by separation).
    In addition, the set of tubes of $G$ also represents the set of facets of $P_G$.
    An $n$-dimensional convex polytope $P$ is called a \emph{Delzant polytope} if the outward normal vectors to the facets meeting at each vertex form an integral basis of $\Z^n$. In the rest of this section, we shall discuss the Delzant construction of $P_G$.

    \begin{lemma}\label{lem:truncation and delzant}
        Let $P$ be a Delzant polytope and $F$ a proper face of $P$. Then there is the canonical truncation of $P$ along $F$ such that the result is a Delzant polytope.
    \end{lemma}
    \begin{proof}
        Assume $\dim P=n$. Suppose that $F$ is the intersection of facets $F_1,\ldots,F_k$ for $1\leq k \leq n$ whose outward normal vectors are $\lambda(F_1),\ldots,\lambda(F_k)$, respectively. We truncate $P$ along $F$ by a hyperplane whose normal vector is $\lambda(F_1) + \cdots + \lambda(F_k)$, and obtain a convex simple polytope $\cut_F(P)$. We claim that $\cut_F(P)$ is Delzant. In order to prove the claim, by the convexity of $\cut_F(P)$, the only thing what we have to show is that for every new vertex $v$ of $\cut_F(P)$, the normal vectors of facets containing $v$ are unimodular.
        If $k=1$, then $\cut_F(P)$ is clearly a Delzant polytope.
        Now assume that $k>1$. Then at each vertex $v$ of $F$, there exist facets $F_{k+1}^v,\ldots,F_n^v$ such that the outward normal vectors $\lambda(F_1)$, $\ldots$, $\lambda(F_k)$, $\lambda(F_{k+1}^v)$, $\ldots$, $\lambda(F_n^v)$ form an integral basis of $\Z^n$. Let $\tilde{F}$ be the facet of $\cut_F(P)$ which is the new facet arising from the truncation of the face $F$. Note that the facets $F_1,\cdots,F_k$ and $F_{k+1}^v,\ldots,F_n^v$ are also cut when we truncate the face $F$ of $P$. Let $\tilde{F_1}$,\ldots,$\tilde{F_k}$, $\tilde F_{k+1}^v$,\ldots,$\tilde F_n^v$ be the facets of $\cut_F(P)$ which are the cuts of $F_1$, $\ldots$, $F_k$, $F_{k+1}^v$,\ldots, $F_n^v$, respectively.  We note that each new vertex of $\tilde{F}$ can be written as the intersection of $\tilde{F}_1$, $\ldots$, $\tilde{F}_{i-1}$, $\tilde{F}$, $\tilde{F}_{i+1}$, $\ldots$, $\tilde{F}_k$, $\tilde{F}_{k+1}^v$, $\ldots$, $\tilde{F}_n^v$, where $v$ is a vertex of $F$. Furthermore, in this case, $\{F_1, \ldots,F_k,F_{k+1}^v,\ldots, F_n^v \}$ should meet at the vertex $v$ in $P$. Since $P$ is Delzant, their normal vectors are unimodular. Therefore, we have
        \begin{equation*}
        \begin{split}
        &\det \begin{pmatrix} \lambda(F_1) & \cdots&\lambda(F_k)&\lambda(F_{k+1}^v\cdots & \lambda(F_n^v)
        \end{pmatrix}\\
        & = \det \begin{pmatrix} \lambda(\tilde{F}_1) & \cdots & \lambda(\tilde{F}_{i-1}) & \lambda(\tilde{F}) & \lambda(\tilde{F}_{i+1}) & \cdots &\lambda(\tilde{F}_k)&\lambda(\tilde{F}_{k+1}^v)&\cdots& \lambda(\tilde{F}_n^v)
        \end{pmatrix},
        \end{split}
        \end{equation*}
        and hence $\cut_F(P)$ is also Delzant. %
    \end{proof}

    Let $G$ be a connected pseudograph with node set $V=[n+1]$ and bundles $B_1, \ldots, B_k$, and let $|B_i|=b_i+1$ for $i=1,\ldots,k$.
    Each edge in a bundle $B_i$ is labelled by $e_i^j$ for $i\in[k]$ and $j \in [b_i+1]$. We set
    $\cC_G = V\cup B_1 \cup \cdots \cup B_k$ and $\cR_G = \cC_G \setminus \{ n+1, e_1^{b_1+1}, \ldots, e_k^{b_k+1}\}$.
    We note that $\Delta_G$ is $|\cR_G|$-dimensional.
    Consider an integral matrix $A =(a_{\alpha,\beta})$ of size $|\cR_G| \times |\cC_G|$ whose rows are labelled by elements in $\cR_G$ and columns are labelled by elements in $\cC_G$.
    Then, there is a Delzant realization of $\Delta_G$ such that the normal vector of the facet corresponding to $\beta \in \cC_G$ is the column $A_\beta$ of $A$ labelled by $\beta$ if we put
    $$
    a_{\alpha,\beta} = \left\{
                            \begin{array}{ll}
                              -1, & \hbox{if $\alpha = \beta$;} \\
                              1, & \hbox{if $\alpha \in V$ and $\beta = n+1$, or $\alpha \in B_i$ and $\beta= e_i^{b_i+1}$;} \\
                              0, & \hbox{otherwise.}
                            \end{array}
                        \right.
    $$
    Since $P_G$ is obtained from $\Delta_G$ by truncating facets canonically, it follows from Lemma~\ref{lem:truncation and delzant} that for a pseudograph $G$,  the pseudograph associahedron $P_G$ becomes a Delzant polytope.
    More precisely, the outward normal vector of the facet corresponding to a tube~$I$ is $\sum_{\beta \in L_I} A_\beta = \sum_{\beta \in I} A_\beta$.
    Let $\lambda_G$ be the mod $2$ characteristic function of $P_G$. Then, $\lambda_G(I) \equiv \sum_{\beta \in I} A_\beta \pmod{2}$.
    For simplicity, let $M_G:=M_{\lambda_G}(P_G)$, which is a real toric manifold associated with $\lambda_G$.

\section{The $a$-polynomial of a pseudograph and the main result} \label{sec3}
    In this section, we will define the $a$-polynomial of a pseudograph
    and explain our main result.

    Let $G$ be a connected pseudograph with node set $V$ and bundles $B_1,\ldots,B_k$, and $\cC_G = V \cup B_1 \cup \cdots \cup B_k$.
    We set
    $$2^{\cC_G}_{\mathrm{even}} = \{ C \subset \cC_G \mid |C\cap V|\equiv|C\cap B_1|\equiv\cdots\equiv|C\cap B_k|\equiv 0 \text{ (mod $2$)}\} \subset 2^{\cC_G}.$$
    A collection $C\subset \cC_G$ is said to be \emph{admissible} to $G$ if it satisfies the following (a1)$\sim$(a3):
    \begin{itemize}
      \item[(a1)] $C\in 2^{\cC_G}_{\mathrm{even}}$,
      \item[(a2)] $C$ contains the nodes which are not endpoints of any bundle of $G$, and
      \item[(a3)] for each bundle $B$ of $G$,  $B\cap C\neq \emptyset$.
    \end{itemize}
    If $G$ has $q$ connected components $G^1$, \ldots, $G^q$, then $C\subset \cC_G$ is said to be \emph{admissible} to~$G$ if $C\cap G^i$ is admissible to $G^i$ for every $i=1,\ldots,q$. If $C\subset \cC_G$ is admissible to $G$, then $C$ is called an \emph{admissible collection} of $G$.
    For example, Figure~\ref{fig:admissible} shows all admissible collections of the pseudograph~$G$ in Figure~\ref{fig:K3} and Figure~\ref{fig:nonadmissible} shows some non-admissible collections.

    \begin{figure}[h]
    \begin{center}
    \begin{subfigure}[b]{0.2\textwidth}
	\centering
    \begin{tikzpicture}[scale=0.9]
    	\fill (3,0) circle(2pt);
    	\fill (0,0) circle(2pt);
    	\fill (1.5,0) circle(2pt);
        \draw (0,0)..controls (0.75,-0.4)..(1.5,0);	
        \draw (3,0)--(1.5,0);
    	\draw (0,0)..controls (0.75,0.4)..(1.5,0);
    	\draw (3.3,0) node{\footnotesize$3$};
    	\draw (-0.3,0) node{ $$};
    	\draw (1.6,0.2) node{\footnotesize$2$};
    	\draw (0.75,0.5) node{\footnotesize$a$};
    	\draw (0.75,-0.5) node{\footnotesize$b$};
    \end{tikzpicture}\caption*{$C_1=23ab$}
    \end{subfigure}\hspace{2cm}
    \begin{subfigure}[b]{0.2\textwidth}
    \centering
     \begin{tikzpicture}[scale=0.9]
    	\fill (3,0) circle(2pt);
    	\fill (0,0) circle(2pt);
    	\fill (1.5,0) circle(2pt);
        \draw (0,0)..controls (0.75,-0.4)..(1.5,0);	
        \draw (3,0)--(1.5,0);
    	\draw (0,0)..controls (0.75,0.4)..(1.5,0);
    	\draw (3.3,0) node{\footnotesize$3$};
    	\draw (-0.3,0) node{\footnotesize$1$};
    	\draw (1.6,0.2) node{ };
    	\draw (0.75,0.5) node{\footnotesize$a$};
    	\draw (0.75,-0.5) node{\footnotesize$b$};
    \end{tikzpicture}\caption*{$C_2=13ab$}
    \end{subfigure}
    \caption{Admissible collections}\label{fig:admissible}
    \end{center}
    \end{figure}
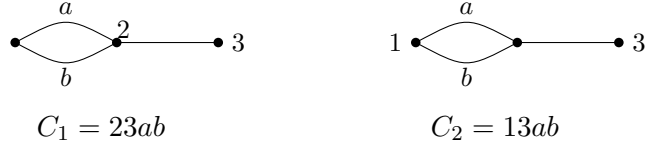

    \begin{figure}[h]
    \begin{center}
    \begin{subfigure}[b]{0.2\textwidth}
    \centering
    \begin{tikzpicture}[scale=0.9]
    	\fill (3,0) circle(2pt);
    	\fill (0,0) circle(2pt);
    	\fill (1.5,0) circle(2pt);
        \draw (0,0)..controls (0.75,-0.4)..(1.5,0);	
        \draw (3,0)--(1.5,0);
    	\draw (0,0)..controls (0.75,0.4)..(1.5,0);
    	\draw (3.3,0) node{\footnotesize$3$};
    	\draw (-0.3,0) node{\footnotesize$1$};
    	\draw (1.6,0.2) node{\footnotesize$2$};
    	\draw (0.75,0.5) node{\footnotesize$a$};
    	\draw (0.75,-0.5) node{\footnotesize$b$};
    \end{tikzpicture}
    \caption*{$C_3=123ab$}
    \end{subfigure}\hspace{0.5cm}
    \begin{subfigure}[b]{0.2\textwidth}
    \centering
    \begin{tikzpicture}[scale=0.9]
    	\fill (3,0) circle(2pt);
    	\fill (0,0) circle(2pt);
    	\fill (1.5,0) circle(2pt);
        \draw (0,0)..controls (0.75,-0.4)..(1.5,0);	
        \draw (3,0)--(1.5,0);
    	\draw (0,0)..controls (0.75,0.4)..(1.5,0);
    	\draw (3.3,0) node{\footnotesize$3$};
    	\draw (-0.3,0) node{};
    	\draw (1.6,0.2) node{};
    	\draw (0.75,0.5) node{\footnotesize$a$};
    	\draw (0.75,-0.5) node{\footnotesize$b$};
    \end{tikzpicture}
    \caption*{$C_4=3ab$}
    \end{subfigure}\hspace{0.5cm}
    \begin{subfigure}[b]{0.2\textwidth}
    \centering
     \begin{tikzpicture}[scale=0.9]
    	\fill (3,0) circle(2pt);
    	\fill (0,0) circle(2pt);
    	\fill (1.5,0) circle(2pt);
        \draw (0,0)..controls (0.75,-0.4)..(1.5,0);	
        \draw (3,0)--(1.5,0);
    	\draw (3.3,0) node{\footnotesize$3$};
    	\draw (-0.3,0) node{ };
    	\draw (1.6,0.2) node{\footnotesize$2$};
    	\draw (0.75,0.5) node{ };
    	\draw (0.75,-0.5) node{\footnotesize$b$};
    \end{tikzpicture}
    \caption*{$C_5=23b$}
    \end{subfigure}
    \caption{Non-admissible collections}\label{fig:nonadmissible}
    \end{center}
    \end{figure}
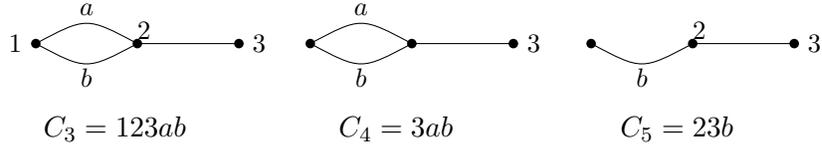

    For a pseudograph $G$ and for each  collection $C\subset \cC_G$, let $P_{C,G}^{\mathrm{odd}}$ denote the union of all facets $I$ of $P_G$ such that $|I\cap C|$ is odd,  and $K_{C,G}^{\mathrm{odd}}$ its dual simplicial complex.

    For a topological space $X$, if $X$ has the homotopy type of a finite CW-complex, then we let $\PP_X(t)$ and $\widetilde{\PP}_X(t)$ be the polynomials defined by
    \[
        \PP_X(t)=\sum_{i=0}^{\infty} \beta^i(X) t^i \quad { \text{ and } } \quad \widetilde{\PP}_X(t)=\sum_{i=0}^{\infty} \tilde{\beta}^i(X) t^i, \text{ respectively},
    \] where $\beta^i(X)=\mathrm{rank}_{\mathbb{Q}} {H}^{i}(X; \Q)$ and $\tilde\beta^i(X)=\mathrm{rank}_{\mathbb{Q}} \widetilde{H}^{i}(X; \Q)$.
    Then we define the \emph{$a$-polynomial} $a_G(t)$ of $G$ by
    \begin{equation}
        a_G(t)=\sum_{C\subset\cC_{G}\atop\text{admissible to }G}\widetilde{\PP}_{K_{C,G}^{\mathrm{odd}}}(t).\label{def:a-poly}
    \end{equation}
    Note that if $G$ has no admissible collection then the $a$-polynomial $a_G(t)$ of $G$ is defined to be the zero polynomial. Hence, if $G$ is a simple graph with odd number of nodes, then $a_G(t)$ is the zero polynomial.
    For example, recall that the pseudograph $G$ in Figure~\ref{fig:K3} has only two admissible collections $C_1$ and $C_2$ in Figure~\ref{fig:admissible}.
    Figure~\ref{fig:a-poly2} says that $K_{C_1,G}^{\mathrm{odd}}$ is homotopy equivalent to $S^1$, and $K_{C_2,G}^{\mathrm{odd}}$ is null-homotopic.
    Thus $\widetilde{\PP}_{K_{C_1,G}^{\mathrm{odd}}}(t)=t$ and $\widetilde{\PP}_{K_{C_2,G}^{\mathrm{odd}}}(t)=0$, and hence the $a$-polynomial of $G$ is $a_G(t)=t$.\label{page:a-poly}

    \begin{figure}[h]
    \centering
    \begin{subfigure}[b]{0.4\textwidth}
    \begin{tikzpicture}
         \path (0,0) coordinate (2)    (2,0) coordinate (3)    (-1.5,0) coordinate (12ab)  (1,1) coordinate (123a)   (1,-1) coordinate (123b);
         \path (2) edge (12ab);
         \path (3) edge (123a);
         \path (3) edge (123b);
         \path (2) edge (123b);
         \path (2) edge (123a);
         \fill (2) node[above]{\tiny$2$}
               (3) node[above]{\tiny$3$}
               (123a) node[above]{\tiny $123a$}
               (123b) node[below]{\tiny $123b$}
               (12ab) node[left]{\tiny $12ab$};
         \fill (2) circle (2pt) (3) circle (2pt)   (12ab) circle (2pt)(123b) circle (2pt) (123a) circle (2pt);
    \end{tikzpicture}
    \caption*{\footnotesize$K^{\mathrm{odd}}_{C_1,G}$}
    \end{subfigure}
    \begin{subfigure}[b]{0.4\textwidth}
     \begin{tikzpicture}
         \path (0,0) coordinate (2)    (2,0) coordinate (3)    (-1.5,0) coordinate (12ab)  (1,1) coordinate (123a)   (1,-1) coordinate (123b);    \draw[fill=red!20]  (2)--(123a)--(3)--(123b)--cycle;
         \path (2) edge (12ab);
         \path (3) edge (123a);
         \path (3) edge (123b);
         \path (2) edge (123b);
         \path (2) edge (123a);
         \path (2) edge (3);
         \fill (2) node[above]{\tiny$1$}
               (3) node[above]{\tiny$3$}
               (123a) node[above]{\tiny $123a$}
               (123b) node[below]{\tiny $123b$}
               (12ab) node[left]{\tiny $12ab$};
         \fill (2) circle (2pt) (3) circle (2pt)   (12ab) circle (2pt)(123b) circle (2pt) (123a) circle (2pt);
    \end{tikzpicture}
    \caption*{\footnotesize$K^{\mathrm{odd}}_{C_2,G}$}
    \end{subfigure}
    \caption{Simplicial complexes corresponding to admissible collections}\label{fig:a-poly2}
    \end{figure}
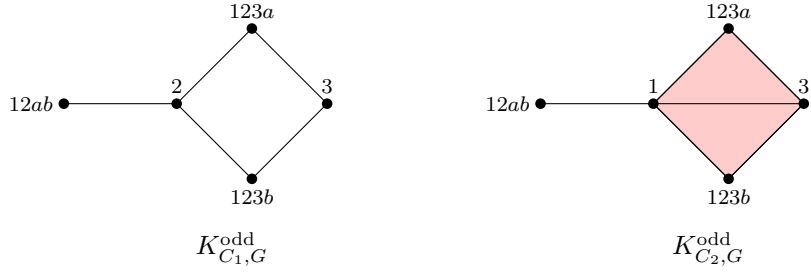

    Note that for a pseudograph $G$, by \eqref{eq:formula}, we already know that the rational Betti numbers of the real toric manifold $M_G$ corresponding to $G$ are computed as
    $$\PP_{M_G}(t)=1+t\sum_{\omega \in \row(\lambda_G)}\widetilde{\PP}_{(P_G)_\omega}(t),$$
    where $\lambda_G$ is the mod $2$ characteristic function of $M_G$ defined at the end of Section~\ref{sec2}.

    We define a matrix $\Lambda' = (\lambda'_{\alpha, I})$ over  $\mathbb{Z}_2$ whose rows are labelled by elements in $\cC_G$ and  columns are labelled by the tubes of $G$ such that
    \[
        \lambda'_{\alpha, I}=\left\{ \begin{array}{ll} 1 &\text{if } \alpha \in I \\
                                               0 & \text{otherwise. }\end{array}  \right.
    \]
    Let $G$ be a connected pseudograph with node set $V=[n+1]$ and bundles $B_1,\ldots,B_k$, where $|B_i|=b_i+1$ for each $i=1,\ldots,k$.
    Then $\lambda_G$ can be obtained from $\Lambda'$ by adding the row $\Lambda'_{n+1}$ to the rows $\Lambda'_{1},\ldots,\Lambda'_{n}$, adding the row $\Lambda'_{e^{b_i+1}_i}$ to the rows $\Lambda'_{e^1_i},\ldots,\Lambda'_{e^{b_i}_i}$ for $i=1,\ldots,k$, and then deleting the rows $\Lambda'_{n+1}$, $\Lambda'_{e^{b_1+1}_1},\ldots,\Lambda'_{e^{b_k+1}_k}$ from $\Lambda'$, where $\Lambda'_{\alpha}$ is the row vector of $\Lambda'$ corresponding to $\alpha \in \mathcal{C}_G$.
    Each subset $C$ of $\cC_G$ is assigned to an element $ \omega_C:=\sum_{c\in C}\Lambda'_{c}$ of $\row(\Lambda')$. Define
    \begin{equation*}
        \row(\Lambda'_{\mathrm{even}}) := \{ \omega_C \in \row(\Lambda') \mid C \in 2^{\cC_G}_{\mathrm{even}}\}.
    \end{equation*}
    Then, $\row(\Lambda'_{\mathrm{even}})$ is a subspace of $\row(\Lambda')$.

    Let $\cR_G = \cC_G \setminus \{ n+1, e_1^{b_1+1}, \ldots, e_k^{b_k+1}\}$.
    For $R \subset \cR_G$, there is a bijection from $2^{\cR_G}$ to $2^{\cC_G}_{\mathrm{even}}$ defined by $R \mapsto C$, where
    \[
        C\cap V=\left\{ \begin{array}{ll} R\cap V &\text{if } |R\cap V| \text{ is even;} \\
        (R\cap V)\cup\{n+1\} & \text{if }|R\cap V| \text{ is odd}, \end{array}  \right.
    \]
    and for  $i=1, \ldots, k$,
    \[
        C\cap B_i=\left\{ \begin{array}{ll} R\cap B_i &\text{if } |R\cap B_i| \text{ is even;} \\
        (R\cap B_i)\cup\{ e_i^{b_i+1}  \} &  \text{if }|R\cap B_i| \text{ is odd}. \end{array}  \right.
    \]

    Note that each element of $\row(\lambda_G)$ is associated with a subset $R$ in $\cR$, and this bijection identifies $\row(\lambda_G)$ with $\row(\Lambda'_{\mathrm{even}})$. Furthermore, one can see that $(P_G)_{ \omega_C}$  is the union of all facets $I$ such that $|I \cap C|$ is odd.
    Hence, by \eqref{eq:formula}, the $i$th rational Betti number $\beta^i(M_G)$ of $M_G$ is given by
    \begin{equation}
        \beta^i(M_G)=\sum_{C \in 2^{\cC_G}_{\mathrm{even}} } \tilde{\beta}^{i-1}(P^{\mathrm{odd}}_{C,G}). \label{lem:P'T}
    \end{equation}
    Hence, for a pseudograph $G$, $$\PP_{M_G}(t)=1+t\sum_{C\in 2^{\cC_G}_{\mathrm{even}}}\widetilde{\PP}_{K_{C,G}^{\mathrm{odd}}}(t).$$

    Recall that a pseudograph $G'$ is a partial underlying pseudograph of $G$ if $G'$ can be obtained from $G$ by replacing some bundles with simple edges, that is, the set of all the bundles of $G'$ is a subset of that of $G$. We denote by $H\lessdot G$ if $H$ is a partial underlying pseudograph of an induced subgraph of $G$.
    See Figure~\ref{fig:semi} for illustrations.
    We now restate the main theorem. The proof will be presented in Section~\ref{sec4}.

\bigskip

    \noindent\textbf{Theorem 1.1}
    \textit{       For any pseudograph $G$,  we have}
        \[\PP_{M_G}(t)=1+t\sum_{H\lessdot G} a_H(t)=1+t\sum_{H\lessdot G}\sum_{C\subset\cC_{H}\atop\text{admissible to $H$}}\widetilde{\PP}_{K_{C,H}^{\mathrm{odd}}}(t).\]
\medskip

    One should note that if $H\lessdot G$, then we have $\cC_H\subseteq \cC_G$ and the dimension of $K_{C,H}^{\mathrm{odd}}$ is much less than the dimension of $K_{C,G}^{\mathrm{odd}}$ for each collection $C\subset\cC_H$. We will see in the next section that for each collection $C$ admissible to $H$, $K_{C,H}^{\mathrm{odd}}$ is a simplicial subcomplex of $K_{C,G}^{\mathrm{odd}}$, and $K_{C,H}^{\mathrm{odd}}$ and $K_{C,G}^{\mathrm{odd}}$ have the same homotopy type.

    Here is an example. Let us compute $\PP_{M_G}(t)$ for the pseudograph $G$ in Figure~\ref{fig:K3}.
    Figure~\ref{fig:semi} shows all pseudographs $H\lessdot G$, 
    where only four pseudographs can
    have an admissible collection, that is, $H_1=123ab$, $H_2=12ab$, $H_3=23$, and $H_4=12$.
    We computed that $a_{H_1}(t)=t$ right after the definition of the $a$-polynomial. 
    The admissible collections of $H_2$ are $12ab$ and $ab$, where $K_{12ab,H_2}^{\mathrm{odd}}$ and $K_{ab,H_2}^{\mathrm{odd}}$ are homotopy equivalent to $S^1$ and $S^0$, respectively. Thus $a_{H_2}(t)=1+t$.
    Since $23$ is the only admissible collection of $H_3$ and $K_{23,H_3}^{\mathrm{odd}}$ is homotopy equivalent to $S^0$, we have $a_{H_3}(t)=1$.
    In addition, $12$ is the only admissible collection of $H_4$ and $K_{12,H_4}^{\mathrm{odd}}$ is homotopy equivalent to $S^0$. Hence, $a_{H_4}(t)=1$.
    Thus $\PP_{M_G}(t)=1+t(t+(1+t)+2)=1+3t+2t^2$.
    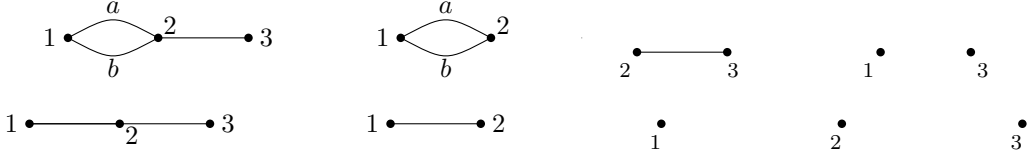
\begin{figure}[h]
    \begin{center}
    \begin{subfigure}[b]{0.3\textwidth}
\hspace{0.8cm}
    \begin{tikzpicture}[scale=0.8]
    	\fill (3,0) circle(2pt);
    	\fill (0,0) circle(2pt);
    	\fill (1.5,0) circle(2pt);
        \draw (0,0)..controls (0.75,-0.4)..(1.5,0);	
        \draw (3,0)--(1.5,0);
    	\draw (0,0)..controls (0.75,0.4)..(1.5,0);
    	\draw (3.3,0) node{\footnotesize$3$};
    	\draw (-0.3,0) node{\footnotesize$1$};
    	\draw (1.7,0.2) node{\footnotesize$2$};
    	\draw (0.75,0.5) node{\footnotesize$a$};
    	\draw (0.75,-0.5) node{\footnotesize$b$};
    \end{tikzpicture}
    \end{subfigure}
    \hspace{0.3cm}
    \begin{subfigure}[b]{0.3\textwidth}

    \vspace{.2cm}
    \end{subfigure}
    \begin{subfigure}[b]{0.2\textwidth}
	 \begin{tikzpicture}[scale=0.8]
    	\fill (0,0) circle(2pt);
    	\fill (1.5,0) circle(2pt);
      	\fill (3,0) circle(0pt);
        \draw (0,0)..controls (0.75,-0.4)..(1.5,0);		
        \draw (0,0)..controls (0.75,0.4)..(1.5,0);
    	\draw (-0.3,0) node{\footnotesize$1$};
    	\draw (1.7,0.2) node{\footnotesize$2$};
    	\draw (0.75,0.5) node{\footnotesize$a$};
    	\draw (0.75,-0.5) node{\footnotesize$b$};
    \end{tikzpicture}
    \end{subfigure}
      \begin{subfigure}[b]{0.2\textwidth}
        \begin{tikzpicture}[scale=0.8]
        \fill (0,0) circle(2pt);
		\fill (1.5,0) circle(2pt);
		\draw (1.5,0)--(0,0)--cycle;
		\draw (-0.2,-0.3) node{{\tiny$2$}};
      	\draw (1.6,-0.3) node{{\tiny$3$}};
	\end{tikzpicture}
    \end{subfigure}
    \begin{subfigure}[b]{.2\textwidth}
    \begin{tikzpicture}[scale=.8]
        \fill (0,0) circle(2pt);
        \fill (1.5,0) circle(2pt);
        \draw (-0.2,-0.3) node{\tiny{$1$}};
        \draw (1.7,-0.3) node{\tiny{$3$}};
    \end{tikzpicture}
    \end{subfigure}\\ \vspace{0.2cm}
    \begin{subfigure}[b]{0.9\textwidth}
    \begin{tikzpicture}[scale=0.8]
 	    \fill (-3,0) circle(2pt);
    	\fill (-6,0) circle(2pt);
    	\fill (-4.5,0) circle(2pt);
        \draw (-6,0)--(-4.5,0);	
        \draw (-6,0)--(-3,0);
    	\draw (-2.7,0) node{\footnotesize$3$};
    	\draw (-6.3,0) node{\footnotesize$1$};
    	\draw (-4.3,-0.2) node{\footnotesize$2$};
    	\fill (0,0) circle(2pt);
    	\fill (1.5,0) circle(2pt);
        \draw (0,0)--(1.5,0);	
    	\draw (-0.3,0) node{\footnotesize$1$};
    	\draw (1.8,0.0) node{\footnotesize$2$};
		\fill (4.5,0) circle(2pt);
		\draw (4.4,-0.3) node{{\tiny$1$}};
		\fill (7.5,0) circle(2pt);
		\draw (7.4,-0.3) node{{\tiny$2$}};
		\fill (10.5,0) circle(2pt);
		\draw (10.4,-0.3) node{{\tiny$3$}};
	\end{tikzpicture}
    \end{subfigure}
    \vspace{.5cm}
    \caption{All pseudographs $H\lessdot G$ for the pseudograph $G$ in Figure~\ref{fig:K3}}\label{fig:semi}
    \end{center}
    \end{figure}

\section{A simplicial complex $K_{C,G}^{\mathrm{odd}}$ and its subcomplexes} \label{sec4}

    In this section, we will show that a simplicial complex $K_{C,G}^{\mathrm{odd}}$ is homotopy equivalent to $K_{C,H}^{\mathrm{odd}}$ for some $H\lessdot G$, 
    and we will prove Theorem~\ref{thm:main}.

    Let $G$ be a connected pseudograph with node set $V = [n+1]$ and bundles $B_1$, \ldots, $B_k$, where $|B_i| = b_i + 1$ for each $i = 1,...,k$, as in the previous section.

    The following lemma is useful to find a simplicial subcomplex which is homotopy equivalent to a given simplicial complex.
    \begin{lemma}[Lemma~5.2 of \cite{CP}] \label{lem:st}
        Let $I$ be a vertex of a simplicial complex  {$K$} and suppose that the link $\Lk I$ of $I$ in $K$ is contractible.
        Then {$K$} is homotopy equivalent to the complex ${K\setminus\St I}$, where $\St I$ is the star of $I$.
    \end{lemma}

    For each $C\in 2^{\cC_G}_{\mathrm{even}}$, let $\widetilde{\Gamma}_G(C)$ be the subgraph of $G$ induced by the node set $V_C:=\{v\in V\mid v\in C \text{ or }v \text{ is an endpoint of an edge }e\in C\}$.
    We define a simplicial subcomplex $K'_{C,G}$ of $K_{C,G}^{\mathrm{odd}}$ by
    $$\{\text{the vertices of }K'_{C,G}\}=\{\text{tubes }I\text{ of }G\text{ such that }I\subseteq \widetilde{\Gamma}_G(C)\text{ and }|I\cap C|\text{ is odd}\}.$$
    We note that the vertices of $K'_{C,G}$ are tubes of $\widetilde{\Gamma}_G(C)$ unless $I= \widetilde{\Gamma}_G(C)$.
    \begin{example}\label{ex2}
        Let $G$ be a connected pseudograph with four nodes and two bundles in Figure~\ref{fig:house graph}.
        Let $C_1=13ab$ and $C_2=12cd$, then
        $\widetilde{\Gamma}_G(C_1)=123ab$ and $\widetilde{\Gamma}_G(C_2)=124abcd$.
        \begin{figure}[h]
        \begin{center}
        \begin{subfigure}[b]{.25\textwidth}
        \centering
        \begin{tikzpicture}
        	\fill (0,0) circle(2pt);
        	\fill (2,0) circle(2pt);
        	\fill (0,2) circle(2pt);
        	\fill (2,2) circle(2pt);
        	\draw (0,0)--(0,2)--(2,2)--(2,0);
        	\draw (0,2)..controls (1,2.5)..(2,2);
        	\draw (2,0)..controls (2.5,1)..(2,2);
        	\draw (-0.3,0) node{\footnotesize$3$};
        	\draw (2.3,0) node{\footnotesize$4$};
        	\draw (-0.3,2) node{\footnotesize$1$};
        	\draw (2.3,2) node{\footnotesize$2$};
        	\draw (1,2.6) node{\footnotesize$a$};
        	\draw (1,1.8) node{\footnotesize$b$};
        	\draw (2.1,1) node{\footnotesize$c$};
        	\draw (2.6,1) node{\footnotesize$d$};
        \end{tikzpicture}
        \caption*{\tiny A pseudograph $G$}
        \end{subfigure}
        \begin{subfigure}[b]{.25\textwidth}
        \centering
        \begin{tikzpicture}
        	\fill (0,0) circle(2pt);
        	\fill (0,2) circle(2pt);
        	\draw (0,0)--(0,2)--(2,2);
        	\draw (0,2)..controls (1,2.7)..(2,2);
        	\draw (-0.3,0) node{\blue{\footnotesize$3$}};
        	\draw (-0.3,2) node{\blue{\footnotesize$1$}};
        	\draw (2.3,2) node{\footnotesize$2$};
        	\draw (1,2.7) node{\blue{\footnotesize$a$}};
        	\draw (1,1.7) node{\blue{\footnotesize$b$}};	
            \filldraw[fill=black] (2,2) circle(2pt);
        \end{tikzpicture}
        \caption*{\tiny   $\widetilde\Gamma_G(13ab)=123ab$.}
          \end{subfigure}
        \hspace{-0.5cm}
        \begin{subfigure}[b]{.25\textwidth}
        \centering
        \begin{tikzpicture}
        	\fill (2,0) circle(2pt);
        	\fill (0,2) circle(2pt);
        	\fill (2,2) circle(2pt);
        	\draw  (2,0)--(2,2)--(0,2);
        	\draw (0,2)..controls (1,2.5)..(2,2);
        	\draw (2,0)..controls (2.5,1)..(2,2);
        	\draw (2.3,0) node{\footnotesize$4$};
        	\draw (-0.3,2) node{\blue{\footnotesize$1$}};
        	\draw (2.3,2) node{\blue{\footnotesize$2$}};
        	\draw (1,2.6) node{\footnotesize$a$};
        	\draw (1,1.8) node{\footnotesize$b$};
        	\draw (2.1,1) node{\blue{\footnotesize$c$}};
        	\draw (2.6,1) node{\blue{\footnotesize$d$}};
        \end{tikzpicture}
        \caption*{\tiny $\widetilde{\Gamma}_G(12cd)=124abcd$.}
        \end{subfigure}
        \hspace{-0.5cm}

        \caption{Examples of $\widetilde{\Gamma}_G$} \label{fig:house graph}
        \end{center}
        \end{figure}
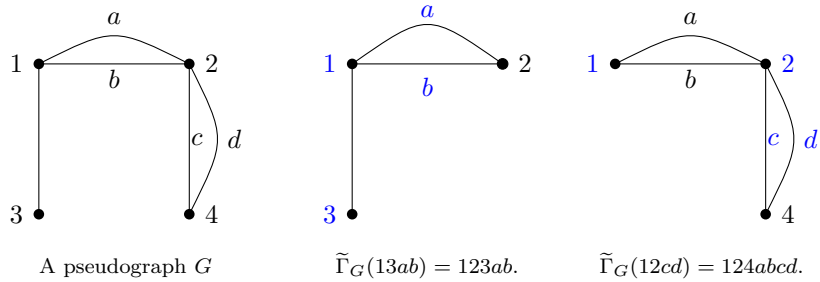

    Note that the vertices of  $K'_{C_1,G}$ are
    \[ 1, 3, 12ab, 123a, \text{ and } 123b,\] and the vertices of $K^{\mathrm{odd}}_{C_1,G}$  are  $124abc$, $124abd$, $124abcd$, $1234ac$, $1234ad$, $1234acd$, $1234bc$, $1234bd$, $1234bcd$, together with vertices of $K'_{C_1,G}$. Figure~\ref{ex:mainex1} shows that $K^{\mathrm{odd}}_{C_1,G}$ and $K'_{C_1,G}$ are homotopy equivalent.

    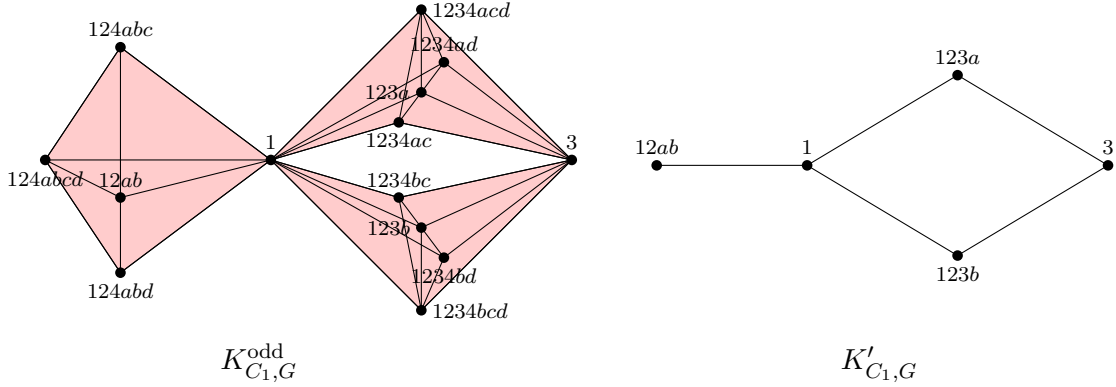
\begin{figure}[h]
    \begin{center}
    \begin{subfigure}[b]{.45\textwidth}
    \centering
    \begin{tikzpicture}
        \path (0,0) coordinate (1)        (-2,-0.5) coordinate (12ab)
            (-3,0) coordinate (124abcd) (-2,1.5) coordinate (124abc)    (-2,-1.5) coordinate (124abd)
            (1.7,0.5) coordinate (1234ac)     (2,0.9) coordinate (123a)       (2,2) coordinate (1234acd)
            (1.7,-0.5) coordinate (1234bc)    (2,-0.9) coordinate (123b)      (2,-2) coordinate (1234bcd)
            (4,0) coordinate (3);
        \path    (2.3,1.3) coordinate (1234ad)    (2.3,-1.3) coordinate (1234bd);
        \draw[fill=red!20]  (1)--(124abc)--(124abcd)--(124abd)--cycle;
        \draw[fill=red!20]  (1)--(1234acd)--(3)--(1234ac)--cycle;
        \draw[fill=red!20]  (1)--(1234bcd)--(3)--(1234bc)--cycle;
        \path (1) edge (12ab) edge (124abc) edge (124abd) edge (1234ac) edge(123a) edge(1234ad) edge(123b) edge(1234bc) edge(1234bd) (124abcd) edge (124abc) edge (12ab) edge (124abd)
              (3)         edge (1234ac) edge(123a) edge(1234ad) edge(123b) edge(1234bc) edge(1234bd)
              (12ab)      edge (124abc) edge (124abd)
              (123a)      edge (1234ad) edge (1234ac)
              (123b)      edge (1234bc) edge (1234bd) edge (1234bcd) edge (123b);
        \path (123a) edge (1234acd);
        \path(1234acd) edge (1);
        \path (1234acd) edge (3);
        \path (1234acd) edge (1234ad);
        \path (1234acd) edge (1234ac);
        \path (1234bcd) edge (1);
        \path (1234bcd) edge (3);
        \path (1234bcd) edge (1234bc);
        \path (1234bcd) edge (1234bd);
        \path (124abcd) edge (1);
        \fill (1) node[above]{\tiny$1$}
              (12ab) node[above]{\tiny $12ab$}
              (124abcd) node[below]{\tiny $124abcd$}
              (124abd) node[below]{\tiny $124abd$}
              (124abc) node[above]{\tiny $124abc$}
              (1234ac) node[below]{\tiny $1234ac$}
              (123a) node[left]{\tiny $123a$}
              (1234ad) node[above]{\tiny $1234ad$}
              (1234bc) node[above]{\tiny $1234bc$}
              (123b) node[left]{\tiny $123b$}
              (1234bd) node[below]{\tiny $1234bd$}
              (3) node[above]{\tiny $3$};
        \fill  (1234bcd) node[right]{\tiny $1234bcd$};
        \fill (1234acd) node[right]{\tiny $1234acd$};
        \fill (1) circle (2pt)
              (12ab) circle (2pt)
              (124abcd) circle (2pt)
              (124abd) circle (2pt)
              (124abc) circle (2pt)
              (1234ac) circle (2pt)
              (123a) circle (2pt)
              (1234ad) circle (2pt)
              (1234bc) circle (2pt)
              (123b) circle (2pt)
              (1234bd) circle (2pt)
              (3) circle (2pt);
        \fill (1234acd) circle (2pt);
        \fill (1234bcd) circle (2pt);
    \end{tikzpicture}
    \caption*{$K^{\mathrm{odd}}_{C_1,G}$}
    \end{subfigure}
    \hspace{1cm}
    \begin{subfigure}[b]{.45\textwidth}
    \centering
    \begin{tikzpicture}
        \path (0,0) coordinate (1)        (-2,0) coordinate (12ab)
              (2,1.2) coordinate (123a)
              (2,-1.2) coordinate (123b)
              (4,0) coordinate (3);
        \path (1) edge (12ab) edge(123a) edge(123b)
              (3) edge(123a) edge(123b);
        \fill (1) node[above]{\tiny$1$}
              (12ab) node[above]{\tiny $12ab$}
              (123a) node[above]{\tiny $123a$}
              (123b) node[below]{\tiny $123b$}
              (3) node[above]{\tiny $3$};
        \fill (1) circle (2pt)
              (12ab) circle (2pt)
              (123a) circle (2pt)
              (123b) circle (2pt)
              (3) circle (2pt);
    \end{tikzpicture}
    \vspace{0.5cm}
    \caption*{$K'_{C_1,G}$}
    \end{subfigure}
    \caption{For the pseudograph $G$ and $C_1=13ab$ in Figure~\ref{fig:house graph}, two simplicial complexes  $K^{\mathrm{odd}}_{C_1,G}$ and $K'_{C_1,G}$ are homotopy equivalent.}\label{ex:mainex1}
    \end{center}
    \end{figure}
    \end{example}
    The following lemma shows that the phenomenon in the above example holds for any $C\in 2^{\cC_G}_{\mathrm{even}}$.

    \begin{lemma}\label{thm:KT}
      $K_{C,G}^{\mathrm{odd}}$ is homotopy equivalent to $K'_{C,G}$.
    \end{lemma}

    \begin{proof}
        It is enough to consider the case when $\widetilde{\Gamma}_G(C)$ is a proper subgraph of $G$. For simplicity, we write $K_G$ and $K'$ instead of $K_{C,G}^{\mathrm{odd}}$ and $K'_{C,G}$, respectively. We will show that we can eliminate the stars of all vertices in $K_G\setminus K'$, one by one, from $K_G$ to $K'$, without changing the homotopy type.

    Suppose that we cannot eliminate the stars of all vertices in $K_G\setminus K'$ one by one, without changing the homotopy type. Then we can obtain a minimal complex $K^\ast\supsetneq K'$, which is obtained by eliminating the stars of some vertices in $K_G\setminus K'$ without changing the homotopy type.
    Then take a vertex $I$ which is minimal in $K^\ast\setminus K'$. Since $I\in K_G$ and $I\not\in K'$, $|I\cap C|$ is odd and $I\setminus \widetilde{\Gamma}_G(C)\neq\emptyset$.
    Since $\widetilde{\Gamma}_G(C)$ is an induced subgraph, $I\setminus \widetilde{\Gamma}_G(C)$ contains a node.
    Let $\widetilde{I}$ be the subgraph of $I$ which is obtained from $I$ by deleting the nodes not in $ \widetilde{\Gamma}_G(C)$.
    Note that from the definitions of $\widetilde{\Gamma}_G(C)$ and $\widetilde{I}$, $I\cap C=\widetilde{I}\cap C$.
    Then we can choose a connected component $\widetilde{I}_0$ of $\widetilde{I}$ such that $|\widetilde{I}_0\cap C|$ is odd. Since $\widetilde{I}_0\subset \widetilde{I} \subset \widetilde{\Gamma}_G(C)$, we conclude that $\widetilde{I}_0\in K'$.

    We will show that any vertex of $\Lk I$ meets with $\widetilde{I}_0$.
Let $L\in \Lk I$ in $K^\ast$. If $I$ and $L$ meet by separation, then $\widetilde{I}_0$ and $L$ also meet by separation. Now assume that $I$ and $L$ meet by inclusion. If $I\subset L$, then $\widetilde{I}_0\subset L$. Suppose that $L\subsetneq I$.
The minimality of $I$ implies that $L\in K'$, and hence $L \subseteq \widetilde{\Gamma}_G(C)$. As $L$ is a connected subgraph contained in $I$ and $I \cap \widetilde{\Gamma}_G(C) = \widetilde{I}$, $L$ must be contained in some connected component of $\widetilde{I}$. Hence, either $L$ is contained in $\widetilde{I}_0$, or $L$ and $\widetilde{I}_0$ meet by separation. Therefore, $\Lk I$ is contractible and $K^\ast$ is homotopy equivalent to $K^\ast\setminus \St I$ by Lemma~\ref{lem:st}. Then $K^\ast\setminus \St I$ is smaller than $K^\ast$, which contradicts the minimality of $K^\ast$.
    \end{proof}

    Let us consider a simplicial subcomplex $K''_{C,G}$ of $K'_{C,G}$ whose vertex set consists of vertices $I$ of $K'_{C,G}$ such that, for each bundle $B$ of $\widetilde{\Gamma}_G(C)$ satisfying $B \cap C = \emptyset$, if the endpoints of $B$ are in $I$, then $B \subset I$. That is, we obtain $K''_{C,G}$ from $K'_{C,G}$ by removing the stars of the vertices $I$ if there exists a bundle $B$ of $\widetilde\Gamma_G(C)$ such that $B\cap C=\emptyset$, $I$ contains the endpoints of $B$, and $B\not\subset I$.

    Back to Example~\ref{ex2}, the vertices of $K''_{C_2,G}$ are \[1,2,24cd, 124abc, \text{ and } 124abd,\] and the vertices of $K'_{C_2,G}$ are $124ac,124bc,124ad,124bd$, together with the vertices of $K''_{C_2,G}$. Figure~\ref{ex:mainex2} shows that $K'_{C_2,G}$ and $K''_{C_2,G}$ are homotopy equivalent.
    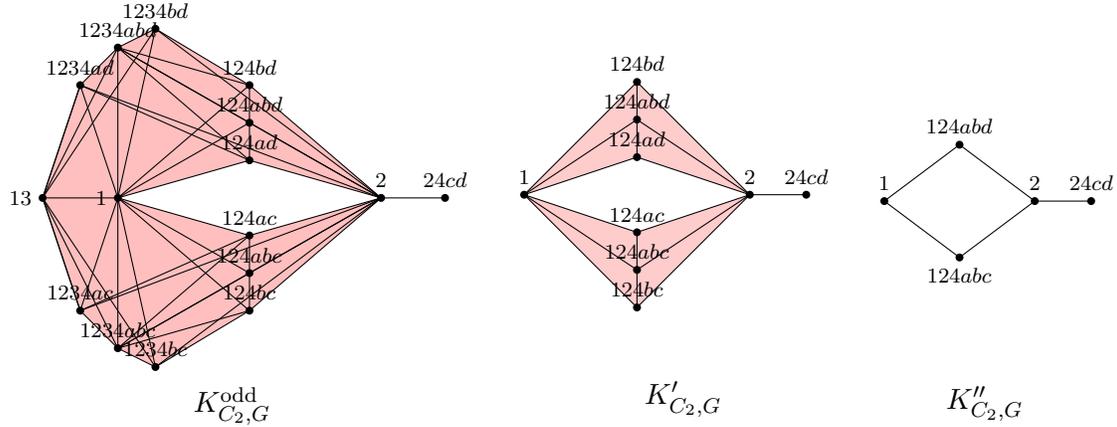
\begin{figure}[h]
    \begin{center}
    \begin{subfigure}{.4\textwidth}
    \begin{tikzpicture}[scale=.5]
        \path
        (-2,0) coordinate (13)
        (-1,3) coordinate (1234ad)
        (0,4) coordinate (1234abd)
        (1,4.5) coordinate (1234bd)
        (-1,-3) coordinate (1234ac)
        (0,-4) coordinate (1234abc)
        (1,-4.5) coordinate (1234bc)
        (3.5,1) coordinate (124ad)
        (3.5,2) coordinate (124abd)
        (3.5,3) coordinate (124bd)
        (3.5,-1) coordinate (124ac)
        (3.5,-2) coordinate (124abc)
        (3.5,-3) coordinate (124bc)
        (0,0) coordinate (1)
        (7,0) coordinate (2)
        (8.7,0) coordinate (24cd);
        \fill[red!25] (13)--(1)--(124ad)--(2)--(124bd)--(1234bd)--(1234abd)--(1234ad)--cycle;
        \fill[red!25] (13)--(1)--(124ac)--(2)--(124bc)--(1234bc)--(1234abc)--(1234ac)--cycle;
        \fill (13) circle (3pt) (1234ad) circle (3pt) (1234abd) circle (3pt) (1234bd) circle (3pt) (1234ac) circle (3pt) (1234abc) circle (3pt) (1234bc) circle (3pt) (124ad) circle (3pt) (124abd) circle (3pt) (124bd) circle (3pt) (124ac) circle (3pt) (124abc) circle (3pt) (124bc) circle (3pt) (1) circle (3pt) (2) circle (3pt) (24cd) circle (3pt);
        \fill
        (1) node[left]{\tiny$1$}
        (13) node[left]{\tiny$13$}
        (1234ad) node[above]{\tiny$1234ad$}
        (1234abd) node[above]{\tiny$1234abd$}
        (1234bd) node[above]{\tiny$1234bd$}
        (1234ac) node[above]{\tiny$1234ac$}
        (1234abc) node[above]{\tiny$1234abc$}
        (1234bc) node[above]{\tiny$1234bc$}
        (124ad) node[above]{\tiny$124ad$}
        (124abd) node[above]{\tiny$124abd$}
        (124bd) node[above]{\tiny$124bd$}
        (124ac) node[above]{\tiny$124ac$}
        (124abc) node[above]{\tiny$124abc$}
        (124bc) node[above]{\tiny$124bc$}
        (2) node[above]{\tiny$2$}
        (24cd) node[above]{\tiny$24cd$};
        \path (13) edge (1) edge (1234abd) edge (1234ad) edge (1234bd) edge (1234ac) edge (1234abc) edge (1234bc);
        \path (1234bd) edge (1) edge (124bd) edge (2);
        \path (1234abd) edge (1) edge (124abd) edge (124ad) edge (124bd) edge (2);
        \path (1234ad) edge (1) edge (124ad) edge (2);
        \path (1234bc) edge (1) edge (124bc) edge (2);
        \path (1234abc) edge (1) edge (124abc) edge (124ac) edge (124bc) edge (2);
        \path (1234ac) edge (1) edge (124ac) edge (2);
        \path (1) edge (124ad) edge (124abd) edge (124bd) edge (124ac) edge (124abc) edge (124bc);
        \path (2) edge (124ad) edge (124abd) edge (124bd) edge (124ac) edge (124abc) edge (124bc) edge (24cd);
        \draw (1234bd)--(1234abd)--(1234ad)--(13)--(1234ac)--(1234abc)--(1234bc);
        \draw (124bd)--(124abd)--(124ad);
        \draw (124ac)--(124abc)--(124bc);
    \end{tikzpicture}
    \caption*{$K_{C_2,G}^{\mathrm{odd}}$}
    \end{subfigure}
    \hspace{0.3cm}
    \begin{subfigure}{.3\textwidth}
    \vspace{0.6cm}
    \begin{tikzpicture}[scale=.5]
        \path
        (3,1) coordinate (124ad)
        (3,2) coordinate (124abd)
        (3,3) coordinate (124bd)
        (3,-1) coordinate (124ac)
        (3,-2) coordinate (124abc)
        (3,-3) coordinate (124bc)
        (0,0) coordinate (1)
        (6,0) coordinate (2)
        (7.5,0) coordinate (24cd);
        \fill[red!20] (1)--(124ad)--(2)--(124bd)--cycle;
        \fill[red!20] (1)--(124ac)--(2)--(124bc)--cycle;
        \fill
        (1) node[above]{\tiny$1$}
        (124ad) node[above]{\tiny$124ad$}
        (124abd) node[above]{\tiny$124abd$}
        (124bd) node[above]{\tiny$124bd$}
        (124ac) node[above]{\tiny$124ac$}
        (124abc) node[above]{\tiny$124abc$}
        (124bc) node[above]{\tiny$124bc$}
        (2) node[above]{\tiny$2$}
        (24cd) node[above]{\tiny$24cd$};
        \fill (124ad) circle (3pt) (124abd) circle (3pt) (124bd) circle (3pt) (124ac) circle (3pt) (124abc) circle (3pt) (124bc) circle (3pt) (1) circle (3pt) (2) circle (3pt) (24cd) circle (3pt);
        \path (1) edge (124ad) edge (124abd) edge (124bd) edge (124ac) edge (124abc) edge (124bc);
        \path (2) edge (124ad) edge (124abd) edge (124bd) edge (124ac) edge (124abc) edge (124bc) edge (24cd);
        \draw (124bd)--(124abd)--(124ad);
        \draw (124ac)--(124abc)--(124bc);
    \end{tikzpicture}
    \vspace{0.7cm}
    \caption*{$K'_{C_2,G}$}
    \end{subfigure}
    \begin{subfigure}{.2\textwidth}
    \vspace{1.5cm}
    \begin{tikzpicture}[scale=.5]
        \path
        (2,1.5) coordinate (124abd)
        (2,-1.5) coordinate (124abc)
        (0,0) coordinate (1)
        (4,0) coordinate (2)
        (5.5,0) coordinate (24cd);
        \fill (124abd) circle (3pt) (124abc) circle (3pt) (1) circle (3pt) (2) circle (3pt) (24cd) circle (3pt);
        \fill
        (1) node[above]{\tiny$1$}
        (124abd) node[above]{\tiny$124abd$}
        (124abc) node[below]{\tiny$124abc$}
        (2) node[above]{\tiny$2$}
        (24cd) node[above]{\tiny$24cd$};
        \path (1) edge (124abd) edge (124abc);
        \path (2) edge (124abd) edge (124abc) edge (24cd);
    \end{tikzpicture}\vspace{1cm}
    \caption*{$K''_{C_2,G}$}
    \end{subfigure}
    \end{center}
    \caption{For $C_2=12cd$, three simplicial complexes $K_{C_2,G}^{\mathrm{odd}}$, $K'_{C_2,G}$, and $K''_{C_2,G}$ are homotopy equivalent.}\label{ex:mainex2}
    \end{figure}

    \begin{lemma}\label{lem:LCodd}
       $K'_{C,G}$ is homotopy equivalent to $K''_{C,G}$.
    \end{lemma}

    \begin{proof} For simplicity, we write $K'$ and $K''$ instead of $K'_{C,G}$ and $K''_{C,G}$, respectively.
    We give a surjective map from the vertices of $K'$ to the vertices of $K''$ so that $I\in K'$ corresponds to $\widetilde{I}\in K''$ where
    $\widetilde{I}$ is a subgraph of $G$ obtained from $I$ by transforming $I\cap B$ (if it is non-empty) into a bundle $B$ whenever there exists a bundle $B$ of $G$ such that  $C\cap B=\emptyset$.
    Note that for any $I \in K'$, $I=\widetilde{I}$ if and only if $I \in K''$.

    Now we will show that we can eliminate the stars of all vertices in $ K' \setminus K''$, one by one, from $K'$ to $K''$, without changing the homotopy type. Suppose we cannot eliminate the stars of all vertices in $K'\setminus K''$ one by one, without changing the homotopy type. Then we can obtain a minimal complex $K^\ast\supsetneq K''$, which is obtained by eliminating the stars of some vertices in $K'\setminus K''$ as long as eliminating does not change the homotopy type.
    Then take a vertex $I$ which is maximal in $K^{\ast}\setminus K''$.
    Then $I\subsetneq \widetilde{I}$ and $\widetilde{I}\in K''$. We will show that $\widetilde{I}$ meets every vertex in $\Lk I$. Take  $L\in\Lk I$. If $I$ and $L$ meet by separation, then so do $\widetilde{I}$ and $L$ because $I$ and $\widetilde{I}$ have the same node set. If $L\subseteq I$, then clearly $L\subseteq\widetilde{I}$. Now assume that $I\subsetneq L$. Then the maximality of $I$ implies $L\in K''$. Therefore $L=\widetilde{L}$.
    Since  $I\subsetneq L$, $\widetilde{I}\subset \widetilde{L}=L$.
    Thus $\widetilde{I}$ and $L$ meet by inclusion.
    Therefore $K^\ast$ is homotopy equivalent to $K^\ast\setminus\St (I)$ by Lemma~\ref{lem:st}.
    Then $K^\ast\setminus \St I$ is smaller than $K^\ast$, which contradicts the minimality of $K^{\ast}$.
    \end{proof}

    By Lemmas~\ref{thm:KT} and \ref{lem:LCodd}, formula~\eqref{lem:P'T} for the $i$th rational Betti number of $M_G$ for a connected pseudograph $G$ can be written as follows:
    \begin{equation}
        \beta^i(M_G) =\sum_{C\in 2^{{\cC}_G}_{\mathrm{even}}} \tilde{\beta}^{i-1}( K''_{C,G} ). \label{cor:main}
    \end{equation}

    Before we give a proof of Theorem~\ref{thm:main}, let us  prepare the following lemma which enables us to determine whether $K_{C,G}^{\mathrm{odd}}$ is contractible or not for a given collection $C\in 2^{\cC_G}_{\mathrm{even}}$.
    For $C\in 2^{\cC_G}$, a subgraph $J$ of a pseudograph $G$ is \emph{odd} (respectively, \emph{even}) with respect to $C$ if $|J\cap C|$ is odd (respectively, even).

    \begin{lemma}\label{thm:ignore_odd}
        For $C\in 2^{\cC_G}_{\mathrm{even}}$, let $\Gamma^1$, $\ldots$, $\Gamma^{q}$ be the connected components of $\widetilde\Gamma_G(C)$, and let $C^j=C\cap \Gamma^j$ for $j=1,\ldots,q$.
        \begin{itemize}
            \item[(i)] If some connected component of $\widetilde\Gamma_G(C)$ is odd with respect to $C$ (equivalently, $|C^{j}|$ is odd for some $j$), then  $K''_{C,G}$  is contractible.
            \item[(ii)] The simplicial complex $K''_{C,G}$ is the simplicial join of all $K''_{C^j,G}$'s for $j=1,\ldots,q$, and hence  $$\widetilde{H}_{i-1}(K''_{C,G};\Q)=\bigoplus_{\sum k_j=i-q}\bigotimes \widetilde{H}_{k_j}(K''_{C^j,G};\Q).$$
        \end{itemize}
    \end{lemma}

    \begin{proof}
        We first assume that some connected component of $\widetilde\Gamma_G(C)$ is odd with respect to $C$.
        Without loss of generality, we may assume that $I^1$ is odd with respect to $C$.
        Then ${I^1}$ is a vertex of $K''_{C,G}$.
        Let $J$ be a tube in $K''_{C,G}$.
        Then $J$ is contained in ${\Gamma^j}$ for some $j$.
        If $j=1$, then
        $J\subset {\Gamma^1}$, and hence ${\Gamma^1}$ and $J$ meet by inclusion.
        If $j\neq 1$ then
        $ {\Gamma^1}$ and $ {J}$ meet by separation.
        Therefore $ {\Gamma^1}$ meets every other vertex of  $K''_{C,G}$. Hence, $K''_{C,G}$ is contractible, which proves the statement (i).

        Note that  $\Gamma^j=\widetilde\Gamma_G(C^j)$ for $j=1,\ldots,q$.
        If $q=1$ then $C=C^1$ and so (ii) is true.
        Suppose that $q\ge 2$.
        Note that each vertex $I$ of $K''_{C,G}$ corresponds to a vertex of $K''_{C^i,G}$ for some $i$.
        Take two vertices $I$ and $J$ of $K''_{C,G}$
        such that $I\subset  {\Gamma^i}$ and $J\subset  {\Gamma^j}$ for $i \neq j$. Then,
        $I$ and $J$ meet by separation, and so the simplex spanned by $I$ and $J$ in $K''_{C,G}$  corresponds to the simplex in $K''_{C^i,G}\ast K''_{C^j,G}$.
        Note that the join $A\ast B$ is homotopy equivalent to the (reduced) suspension of the smash product $A \wedge B$, and $\Sigma(A\wedge B)=S^1\wedge A\wedge B$. Hence, we have
        \begin{equation*}
            \begin{split}
                K''_{C,G}&=K''_{C^1,G}\ast \cdots \ast K''_{C^q,G}\\
                &\simeq\underbrace{S^1\wedge\cdots\wedge S^1}_{q-1}\wedge K''_{C^1,G}\wedge\cdots\wedge K''_{C^q,G}\\
                &=S^{q-1}\wedge K''_{C^1,G}\wedge\cdots\wedge K''_{C^q,G}.
            \end{split}
        \end{equation*}
        Therefore, we have
        $$\widetilde{H}_{i-1}(K''_{C,G};\Q)=\bigoplus_{\sum k_j=i-q}\bigotimes \widetilde{H}_{k_j}(K''_{C^j,G};\Q)$$
        as desired in (ii).
    \end{proof}

    Let $2^{\cC_G}_{\mathrm{even}\ast}$ be
    the set of collections $C$ in $2^{\cC_G}_{\mathrm{even}}$ such that the intersection of $C$ and each of the connected components of $\widetilde\Gamma_G(C)$ belongs to $2^{\cC_G}_{\mathrm{even}}$, that is,
    $$2^{\cC_G}_{\mathrm{even}\ast} = \{ C \in 2^{\cC_G}_{\mathrm{even}} \mid \text{each connected component of }\widetilde\Gamma_G(C)\text{ is even with respect to }C \}.$$
    Due to Lemma~\ref{thm:ignore_odd}, equation~\eqref{cor:main} is equivalent to
    \begin{equation}\label{cor:main1}
        \beta^i(M_G)=\sum_{C\in 2^{\cC_G}_{\mathrm{even}\ast}} \tilde{\beta}^{i-1}( K''_{C,G} ).
    \end{equation}

    Now, we obtain a pseudograph $\Gamma_G(C)$ from $\widetilde\Gamma_G(C)$ by replacing each bundle $B$ of $\widetilde\Gamma_G(C)$ satisfying $C\cap B=\emptyset$ by a (unlabelled) simple edge. Note that $\Gamma_G(C)$ is a partial underlying pseudograph of the induced subgraph $\widetilde\Gamma_G(C)$ of $G$, and hence $\Gamma_G(C)\lessdot G$.

     \begin{example}\label{ex3}
    Recall the pseudograph $G$ in Figure~\ref{fig:house graph}. Then $\Gamma_G(13ab)=123ab$ and $\Gamma_G(12cd)=124cd$. See Figure~\ref{fig:Gamma_G}.
        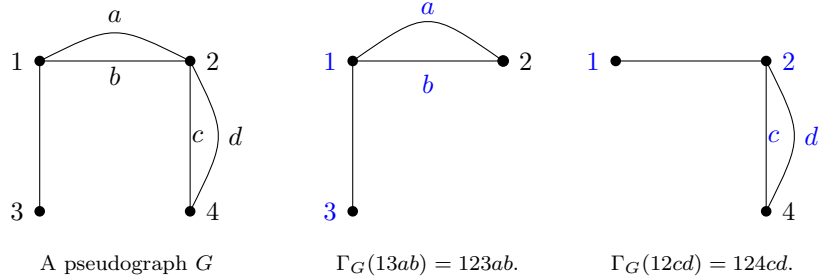
\begin{figure}[h]
        \begin{center}
        \begin{subfigure}[b]{.25\textwidth}
        \centering
        \begin{tikzpicture}
        	\fill (0,0) circle(2pt);
        	\fill (2,0) circle(2pt);
        	\fill (0,2) circle(2pt);
        	\fill (2,2) circle(2pt);
        	\draw (0,0)--(0,2)--(2,2)--(2,0);
        	\draw (0,2)..controls (1,2.5)..(2,2);
        	\draw (2,0)..controls (2.5,1)..(2,2);
        	\draw (-0.3,0) node{\footnotesize$3$};
        	\draw (2.3,0) node{\footnotesize$4$};
        	\draw (-0.3,2) node{\footnotesize$1$};
        	\draw (2.3,2) node{\footnotesize$2$};
        	\draw (1,2.6) node{\footnotesize$a$};
        	\draw (1,1.8) node{\footnotesize$b$};
        	\draw (2.1,1) node{\footnotesize$c$};
        	\draw (2.6,1) node{\footnotesize$d$};
        \end{tikzpicture}
        \caption*{\tiny A pseudograph $G$}
        \end{subfigure}
        \begin{subfigure}[b]{.25\textwidth}
        \centering
        \begin{tikzpicture}
        	\fill (0,0) circle(2pt);
        	\fill (0,2) circle(2pt);
        	\draw (0,0)--(0,2)--(2,2);
        	\draw (0,2)..controls (1,2.7)..(2,2);
        	\draw (-0.3,0) node{\blue{\footnotesize$3$}};
        	\draw (-0.3,2) node{\blue{\footnotesize$1$}};
        	\draw (2.3,2) node{\footnotesize$2$};
        	\draw (1,2.7) node{\blue{\footnotesize$a$}};
        	\draw (1,1.7) node{\blue{\footnotesize$b$}};	
            \filldraw[fill=black] (2,2) circle(2pt);
        \end{tikzpicture}
        \caption*{\tiny   $\Gamma_G(13ab)=123ab$.}
          \end{subfigure}
        \hspace{-0.5cm}
        \begin{subfigure}[b]{.25\textwidth}
        \centering
        \begin{tikzpicture}
        	\fill (2,0) circle(2pt);
        	\fill (0,2) circle(2pt);
        	\fill (2,2) circle(2pt);
        	\draw  (2,0)--(2,2)--(0,2);
        	\draw (2,0)..controls (2.5,1)..(2,2);
        	\draw (2.3,0) node{\footnotesize$4$};
        	\draw (-0.3,2) node{\blue{\footnotesize$1$}};
        	\draw (2.3,2) node{\blue{\footnotesize$2$}};
        	\draw (2.1,1) node{\blue{\footnotesize$c$}};
        	\draw (2.6,1) node{\blue{\footnotesize$d$}};
        \end{tikzpicture}
        \caption*{\tiny $\Gamma_G(12cd)=124cd$.}
        \end{subfigure}
        \hspace{-0.2cm}
        \caption{Examples of $\Gamma_G$} \label{fig:Gamma_G}
        \end{center}
        \end{figure}
    \end{example}

    \begin{proposition}\label{prop:LC}
        For $C\in 2^{\cC_G}_{\mathrm{even}\ast}$, $K^{\mathrm{odd}}_{C,G}$ is homotopy equivalent to $K^{\mathrm{odd}}_{C,\Gamma_G(C)}$, and hence, the $i$th rational Betti number of $M_G$ is
        \begin{equation}\label{cor:main2}
            \beta^i(M_G)=\sum_{C\in 2^{\cC_G}_{\mathrm{even}\ast}} \tilde{\beta}^{i-1}(K^{\mathrm{odd}}_{C,\Gamma_G(C)}).
        \end{equation}
    \end{proposition}
    \begin{proof}
    By equation~\eqref{cor:main1}, it is sufficient to show that $K''_{C,G}$ and $K^{\mathrm{odd}}_{C,\Gamma_G(C)}$ are the same for $C\in 2^{\cC_G}_{\mathrm{even}\ast}$. Figure~\ref{diagram:prop} shows a diagram of this proof.

    \begin{figure}[h]
    \begin{tikzpicture}[scale=1]
    \draw (0,0) node{$K^{\mathrm{odd}}_{C,\Gamma_G(C)}$} (5,0) node{$K''_{C,\widetilde\Gamma_G(C)}$} (8,0) node{$K''_{C,G}$} (0,2) node{$K^{\mathrm{odd}}_{C,\widetilde\Gamma_G(C)}$} (5,2) node{$K'_{C,\widetilde\Gamma_G(C)}$} (8,2) node{$K'_{C,G}$};
    \draw[<->] (1,0)--(4,0);
    \draw (5,1) node{$\cup$};
    \draw (8,1) node{$\cup$};
    \draw (2.5,2) node{$=$} (6.5,2) node{$=$} (6.5,0) node{$=$};
    \draw (2.5,0) node[above]{\footnotesize$\exists$ a bijection $f$} (2.5,0) node[below]{\tiny\text{preserving faces}};
    \end{tikzpicture}
    \caption{A diagram for the proof of Proposition~\ref{prop:LC}, where $C\in 2^{\cC_G}_{\mathrm{even}\ast}$}\label{diagram:prop}
    \end{figure}

    We first claim that $K'_{C,G}$ is equal to $K^{\mathrm{odd}}_{C,\widetilde\Gamma_G(C)}$ for $C\in 2^{\cC_G}_{\mathrm{even}\ast}$. Recall that the vertices of $K'_{C,G}$ are tubes $I$ of $G$ such that $I$ is odd with respect to $C$ and $I\subseteq \widetilde\Gamma_G(C)$. By the connectedness of a tube, if $I$ is a vertex of $K'_{C,G}$, then $I$ is contained in some connected component of $\widetilde\Gamma_G(C)$ and $I$ is odd with respect to~$C$. If $C\in 2^{\cC_G}_{\mathrm{even}\ast}$, each connected component of $\widetilde\Gamma_G(C)$ is even with respect to~$C$ by definition. Hence, $I$ is properly contained in some connected component of $\widetilde\Gamma_G(C)$, which implies that $I$ is also a tube of $\widetilde\Gamma_G(C)$ such that $I$ is odd with respect to $C$. Therefore, $I$ is also a vertex of $K^{\mathrm{odd}}_{C,\widetilde\Gamma_G(C)}$, which implies that $K'_{C,G}$ is a subcomplex of $K^{\mathrm{odd}}_{C,\widetilde\Gamma_G(C)}$. Since $\widetilde\Gamma_G(C)$ is an induced subgraph of $G$, if $I$ is a tube of $\widetilde\Gamma_G(C)$ such that $I$ is odd with respect to~$C$, then it is also a tube of $G$ such that $I\subsetneq\widetilde\Gamma_G(C)$ and $I$ is odd with respect to~$C$. Therefore, every vertex of $K^{\mathrm{odd}}_{C,\widetilde\Gamma_G(C)}$ is also a vertex of $K'_{C,G}$, which implies that $K^{\mathrm{odd}}_{C,\widetilde\Gamma_G(C)}$ is also a subcomplex of $K'_{C,G}$. This proves the claim.

    Note that we obtain $K''_{C,G}$ (respectively $K''_{C,\widetilde\Gamma_G(C)}$) from $K'_{C,G}$ (respectively, $K'_{C,\widetilde\Gamma_G(C)}$) by removing the stars of the vertices $I$ if there exists a bundle $B$ of $\widetilde\Gamma_G(C)$ such that $B\cap C=\emptyset$, $I$ contains the endpoints of $B$, and $B\not\subset I$. Since $K'_{C,G}=K^{\mathrm{odd}}_{C,\widetilde\Gamma_G(C)}=K'_{C,\widetilde\Gamma_G(C)}$ for $C\in 2^{\cC_G}_{\mathrm{even}\ast}$, we can see that $K''_{C,G}=K''_{C,\widetilde\Gamma_G(C)}$.

    From a tube $I$ of $\Gamma_G(C)$, we can find a tube $\widetilde{I}$ of $\widetilde\Gamma_G(C)$ such that $I$ is a partial underlying pseudograph of $\widetilde{I}$ as follows; for each bundle $B$ of $\widetilde\Gamma_G(C)$ satisfying $B\cap C=\emptyset$, whenever $I$ has an unlabelled simple edge $e$ whose endpoints are the same with the endpoints of $B$, we transform such a simple edge $e$ into a bundle $B$. Then $|I\cap C|\equiv |\widetilde{I}\cap C|\pmod{2}$, and hence if $I$ is a vertex of $K^{\mathrm{odd}}_{C,\Gamma_G(C)}$, then $\widetilde{I}$ is also a vertex of $K''_{C,\widetilde\Gamma_G(C)}$. Now we define a map $f$ from the vertex set of $K^{\mathrm{odd}}_{C,\Gamma_G(C)}$ to the vertex set of $K''_{C,\widetilde\Gamma_G(C)}$ by $I\mapsto \widetilde{I}$. Let us show that $f$ is bijective and preserves faces.

    \underline{(i) $f$ is one-to-one.}
    If $I$ and $I'$ are different vertices of $K^{\mathrm{odd}}_{C,\Gamma_G(C)}$, then they have different node sets or different labelled edges. From the definition of the map $f$, $\widetilde{I}$ and $\widetilde{I'}$ also have different node sets or different labelled edges. Hence, $f(I)$ and $f(I')$ are different vertices of $K''_{C,\widetilde\Gamma_G(C)}$.

    \underline{(ii) $f$ is surjective.}
    Let $J$ be a vertex of $K''_{C,\widetilde\Gamma_G(C)}$. If $J$ does not have a bundle $B$ of $\widetilde\Gamma_G(C)$ satisfying $B\cap C=\emptyset$, then $J$ is also an odd tube of $\Gamma_G(C)$ with respect to $C$, which implies that $f(J)=J$. If $J$ has an bundle $B$ of $\widetilde\Gamma_G(C)$ satisfying $B\cap C=\emptyset$, then we can obtain an odd tube $I$ of $\Gamma_G(C)$ with respect to~$C$ from $J$ by transforming such bundles $B$ into unlabelled simple edges. Furthermore, $\widetilde{I}=J$, that is, $f(I)=J$.

    \underline{(iii) $f$ preserves faces.} Take two vertices $I$ and $I'$ in $K^{\mathrm{odd}}_{C,\Gamma_G(C)}$. Since $I$ and $\widetilde{I}$ have the same node set, if $I$ and $I'$ meet by separation, then $\widetilde{I}$ and $\widetilde{I'}$ also meet by separation clearly. Now assume that $I$ and $I'$ meet by inclusion. Without loss of generality, we may further assume that $I\subseteq I'$. Hence, if $I$ has an unlabelled simple edge $e$ whose endpoints are the same with the endpoints of $B$ for some bundle $B$ of $\widetilde\Gamma_G(C)$ satisfying $B\cap C=\emptyset$, then $I'$ also has the edge $e$. Hence, $\widetilde{I}\subseteq \widetilde{I'}$. Thus, if the vertices $I_1,\ldots, I_k$ form a tubing in $\Gamma_G(C)$, then $\widetilde{I_1},\ldots,\widetilde{I_k}$ also form a tubing in $\widetilde\Gamma_G(C)$. Consequently, if the vertices $I_1,\ldots,I_k$ form a simplex in $K^{\mathrm{odd}}_{C,\Gamma_G(C)}$, then $f(I_1),\ldots, f(I_k)$ also form a simplex in $K''_{C,\widetilde\Gamma_G(C)}$.

    Therefore, $K^{\mathrm{odd}}_{C,\Gamma_G(C)}$ is isomorphic to $K''_{C,\widetilde\Gamma_G(C)}$.

    \end{proof}

    \begin{remark} The simplicial subcomplex $K_{C,\Gamma_G(C)}^{\mathrm{odd}}$ is an optimal simplicial subcomplex of $K_{C,G}^{\mathrm{odd}}$ having the same homotopy type, in a sense that $K_{C,H}^{\mathrm{odd}}$ might not have the same homotopy type with  $K_{C,G}^{\mathrm{odd}}$ when $H$ is a proper subgraph of $\Gamma_{G}(C)$.
    Let $G$ be a complete graph with node set $\{1,2,3,4\}$ and one bundle $B=\{a,b,c\}$, where the endpoints of $B$ are $1$ and $2$.
    Then $\Gamma_G(1234ab)=G$. Let us consider a semi-induced subgraph $H$ of $G$ which has four nodes $\{1,2,3,4\}$ and a bundle $B'=\{a,b\}$. Then the simplicial complex $K_{C,H}^{\mathrm{odd}}$ is a simplicial subcomplex of $K_{C,G}^{\mathrm{odd}}$.
    The simplicial complex $K_{C,G}^{\mathrm{odd}}$ (respectively, $K_{C,H}^{\mathrm{odd}}$) is the order complex of a poset consisting of odd tubes of $G$ (respectively, $H$) with respect to $C$, ordered by inclusion. We can compute the reduced Euler characteristic of the order complex of a poset by computing the M\"{o}bius function of the poset. Then one can see that
    the reduced Euler characteristics of $K_{C,G}^{\mathrm{odd}}$ and $K_{C,H}^{\mathrm{odd}}$ are $5$ and $1$, respectively. Hence,
    $K_{C,G}^{\mathrm{odd}}$ is not homotopy equivalent to $K_{C,H}^{\mathrm{odd}}$.
    \end{remark}

    Recall that a collection $C\subset \cC_G$ is admissible to a connected pseudograph $G$ if it satisfies the following (a1)$\sim$(a3):
    \begin{itemize}
      \item[(a1)] $C\in 2^{\cC_G}_{\mathrm{even}}$,
      \item[(a2)] $C$ contains the nodes which are not endpoints of any bundle of $G$, and
      \item[(a3)] for each bundle $B$ of $G$,  $B\cap C\neq \emptyset$.
    \end{itemize}
    If $G$ has $q$ connected components $G^1$, \ldots, $G^q$, then $C\subset \cC_G$ is admissible to $G$ if $C\cap G^i$ is admissible to $G^i$ for $i=1,\ldots,q$.

    \begin{lemma}\label{lem:admissible}
        A collection $C$ belongs to $2^{\cC_G}_{\mathrm{even}\ast}$ if and only if $C$ is admissible to $\Gamma_G(C)$.
    \end{lemma}
    \begin{proof}
        Let us show the forward direction. Assume that  $C\in  2^{{\cC}_G}_{\mathrm{even}\ast}$.
        Put $H=\Gamma_G(C)$, and $H^1$, \ldots, $H^q$ are the connected components of $H$.
        Let $C^i=C\cap H^i$ for each $i=1,\ldots,q$.
        Then $\Gamma_G(C^i)=H^i$.
        Hence, it is enough to show that each $C^i$ is admissible to $H^i$.
        By the definition of $2^{{\cC}_G}_{\mathrm{even}\ast}$, each $C^i$ belongs to $2^{\cC_{H^i}}_{\mathrm{even}}$, and so (a1) holds.
        If $v$ is a node in $H^i\setminus C$, then by definition of $\widetilde\Gamma_G(C)$, $v$ is an endpoint of some bundle of $G$, and so (a2) holds.
        For any bundle $B$ of $H^i$, by definitions of $\widetilde\Gamma_G(C)$ and $\Gamma_G(C)$, $B\cap C\neq \emptyset$, and so (a3) holds.
        Hence, $C^i$ is admissible to $H^i$ as desired.

        The backward direction is clear from (a1) in the definition of admissibility.
    \end{proof}

    By the above lemma, for $C\in 2^{\cC_G}_{\mathrm{even}\ast}$, $\Gamma_G(C)$ has an admissible collection $C$. The lemma below says that the converse also holds.
    \begin{lemma}\label{lem:gamma}
        If $H\lessdot G$ and a collection $C$ is admissible to $H$, then  $H=\Gamma_G(C)$.
    \end{lemma}
    \begin{proof}
        For a node $v$ of $H$, by (a2) and (a3),
        either $v\in C$ or $v$ is an endpoint of some bundle $B$ of $H$ such that $B\cap C\neq \emptyset$, which is exactly equal to the node set of $\Gamma_G(C)$.
        In addition, by (a3) and the definition of a partial underlying pseudograph, the bundles of $H$ are the bundles of $G$ having a nonempty intersection with $C$, which exactly coincide with the bundles of $\Gamma_G(C)$. Therefore, $H =\Gamma_G(C)$.
      \end{proof}

    Now we are ready to give a proof of our main result, Theorem~\ref{thm:main}.

    \begin{proof}[Proof of Theorem~\ref{thm:main}]
        For a connected pseudograph $G$, equation~\eqref{cor:main2} in Proposition~\ref{prop:LC} implies
        \begin{eqnarray*}
        \PP_{M_G}(t) &=&1+t \sum_{C\in 2^{{\cC}_G}_{\mathrm{even}\ast}} \widetilde{\PP}_{K_{C,\Gamma_G(C)}^{\mathrm{odd}}}(t).
        \end{eqnarray*}
        By Lemmas~\ref{lem:admissible} and~\ref{lem:gamma},
        \begin{equation*}
        \begin{array}{rcl}
           2^{{\cC}_G}_{\mathrm{even}\ast}&=&\{C\subset\mathcal{C}_G\mid C\text{ is admissible to }\Gamma_G(C)\}\\
           ~&=& \dot{\cup}_{H\lessdot G}\{C\subset\mathcal{C}_H\mid C\text{ is admissible to }H\},
        \end{array}
        \end{equation*}
        where $\dot{\cup}$ means the disjoint union.
        Thus we have
        \[
        \PP_{M_G}(t) =1+t\sum_{H\lessdot G}\sum_{C\subset\cC_{H}\atop\text{admissible to }H}\hspace{-0.3cm}\widetilde{\PP}_{K_{C,H}^{\mathrm{odd}}}(t)=1+t \sum_{H\lessdot G}a_H(t).
        \]
        If $G$ has $q$ connected components $G^1$, $\ldots$, $G^q$, then $P_G=P_{G^1}\times \cdots\times P_{G^q}$ and hence $M_{G}=M_{G^1}\times\cdots\times M_{G^q}$. Therefore,  
        \begin{eqnarray*}
            \PP_{M_G}(t)&=&\prod_{i=1}^q \PP_{M_{G_i}}(t)=\prod_{i=1}^q\left(1+t\sum_{H^i\lessdot G^i}a_{H^i}(t)\right)=1+t^q\sum_{H^1\lessdot G^1,\ldots,H^q\lessdot G^q}\prod_{i=1}^q a_{H^i}(t).
        \end{eqnarray*}
        By Lemma~\ref{thm:ignore_odd},
        \begin{eqnarray*}
                \PP_{M_G}(t)
        =1+t\sum_{H^1\lessdot G^1,\ldots,H^q\lessdot G^q} a_{H^1\sqcup\cdots\sqcup H^q}(t)=1+t\sum_{H\lessdot G}a_H(t).
        \end{eqnarray*}
    \end{proof}

\section{Further remarks on shellability} \label{sec6:remark}

    So far, we have discussed how to compute the rational Betti numbers of the real toric manifold corresponding to a pseudograph associahedron. The amount of the computation of equation~\eqref{eq:formula} can be reduced for the pseudograph associahedron associated with a pseudograph $G$, since we showed that it is sufficient to consider $K_{C,H}^{\mathrm{odd}}$ for only admissible collections $C\in 2^{\cC_G}_{\mathrm{even}\ast}$ of the pseudographs $H\lessdot G$,
    instead of considering  $P_{C,G}^{\mathrm{odd}}$ for all $C\in 2^{\cC_G}_{\mathrm{even}}$.
    Even though our results help to recognize the homotopy type of $P_{C,G}^{\mathrm{odd}}$ by using $K_{C,H}^{\mathrm{odd}}$ for some special pseudograph $H$, however, we do not have the rich information of the homotopy type of $K_{C,H}^{\mathrm{odd}}$ in general.

    We note that it is shown in~\cite{CP} that if $G$ is a simple graph then for any $C\in 2^{\cC_G}_{\mathrm{even}}$, $K_{C,G}^{\mathrm{odd}}$ and its complement $K_{C,G}^{\mathrm{even}}$ are homotopy equivalent to the order complexes of some posets, respectively, and then they showed that $K_{C,G}^{\mathrm{odd}}$ is homotopy equivalent to a wedge of spheres of the same dimension by showing that the corresponding order complex of  $K_{C,G}^{\mathrm{even}}$ is pure and shellable. Therefore it is natural to ask if similar phenomena occur for pseudographs. The former part, the order complex, is naturally extended to pseudographs, but the latter part, the shellability, is not.

\subsection*{Order complex}
    Let $G$ be a connected pseudograph, and let a collection $C\subset\cC_G$ be admissible to $G$.
    We will see that $K_{C,G}^{\mathrm{odd}}$ is homotopy equivalent to the order complex of the poset $S_{C,G}^{\mathrm{odd}}$ defined in the following.
    Let $S_{C,G}^{\mathrm{odd}}$ (respectively, $S_{C,G}^{\mathrm{even}}$) be the set of all subgraphs $I$ of $G$ such that each connected component of $I$ is an odd (respectively, even) tube with respect to $C$, excluding both $\emptyset$ and $C$.
    For example, consider an admissible collection $C_2$ in Figure~\ref{fig:admissible}.
    Then the vertices of $K_{C_2,G}^{\mathrm{odd}}$ are $1,3, 12ab, 123a, 123b$ and the elements of
    $S_{C_2,G}^{\mathrm{odd}}$ are $1, 3, 12ab, 123a, 123b, 13$.
    Figure~\ref{fig:KS2} shows the order complex of $S_{C_2,G}^{\mathrm{odd}}$.
    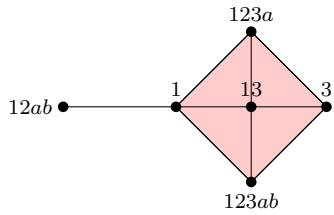
\begin{figure}[h]
    \centering
    \begin{tikzpicture}
        \path
        (0,0) coordinate (2)  (1,0) coordinate (13)  (2,0) coordinate (3)    (-1.5,0) coordinate (12ab)  (1,1) coordinate (123a)   (1,-1) coordinate (123b);
        \draw[fill=red!20]  (2)--(123a)--(3)--(123b)--cycle;
        \path (2) edge (12ab);
        \path (3) edge (123a);
        \path (3) edge (123b);
        \path (2) edge (123b);
        \path (2) edge (123a);
        \path (2) edge (3);
        \path (13) edge (123a) edge (123b);
        \fill (2) node[above]{\tiny$1$}
              (3) node[above]{\tiny$3$}
              (123a) node[above]{\tiny $123a$}
              (123b) node[below]{\tiny $123ab$}
              (13) node[above]{\tiny $13$}
              (12ab) node[left]{\tiny $12ab$};
        \fill (2) circle (2pt) (3) circle (2pt)   (12ab) circle (2pt)(123b) circle (2pt) (123a) circle (2pt) (13) circle (2pt);
    \end{tikzpicture}
    \caption{The order complex of $S_{C_2,G}^{\mathrm{odd}}$ for an admissible collection $C_2=13ab$}\label{fig:KS2}
    \end{figure}

    Then we can observe that when $C$ is an admissible collection of $G$,
    the order complex of $S_{C,G}^{\mathrm{odd}}$ is a geometric subdivision of $K^{\mathrm{odd}}_{C,G}$,
    because the following are true for a vertex $I$ of the order complex  of $S_{C,G}^{\mathrm{odd}}$.
    \begin{itemize}
        \item[(i)] If $I$ is connected, then $I$ is also a vertex of $K^{\mathrm{odd}}_{C,G}$.
        \item[(ii)] If $I$ consists of connected components ${I}^1,\ldots,{I}^\ell$, then ${I}^1,\ldots,{I}^\ell$ are vertices of $K^{\mathrm{odd}}_{C,G}$ and they meet by separation.
    \end{itemize}
    Since a simplicial complex and its geometric subdivision are homotopy equivalent, $K_{C,G}^{\mathrm{odd}}$ and the order complex of $S_{C,G}^{\mathrm{odd}}$ are homotopy equivalent.
    Hence,
    it follows that
    \begin{equation*}
        a_G(t)=\sum_{C\subset\cC_G\atop\text{admissible to }G} \PP_{K_{C,G}^{\mathrm{odd}}}(t)=\sum_{C\subset\cC_G\atop\text{admissible to }G} \PP_{S_{C,G}^{\mathrm{odd}}}(t).
    \end{equation*}
    We note that this gives another definition of the $a$-polynomial, which might be more useful to compute the $a$-polynomial when $S_{C,G}^{\mathrm{odd}}$'s are shellable.

    The order complex  of $S_{C,G}^{\mathrm{even}}$ is a geometric subdivision of $K_{C,G}^{\mathrm{even}}$ if we define $K_{C,G}^{\mathrm{even}}$  by the dual simplicial complex of the union of facets corresponding to the tubes of $G$ that are  even with respect to $C$.
    Hence, since the rational Betti numbers of $K^{\mathrm{odd}}_{C,G}$ are obtained by computing the rational Betti numbers of  $K^{\mathrm{even}}_{C,G}$   by Alexander duality, one might obtain $a_G(t)$ by computing $\PP_{K_{C,G}^{\mathrm{even}}}(t)$ or $\PP_{S_{C,G}^{\mathrm{even}}}(t)$.

\subsection*{Shellability}
    A simplicial complex $K$ is \emph{shellable} if its facets can be arranged in linear order $F_1, \ldots ,F_t$ in such a way that the subcomplex $(\sum_{i=1}^{k-1} \overline{F_i})\cap\overline{F_k}$ is pure and $(\dim {F_k}-1)$-dimensional for all $k = 2, \ldots, t$. For details, see~\cite{BW1996}.

    As we mentioned above, it is shown in~\cite{CP} that if $G$ is a simple graph then for any $C\in 2^{\cC_G}_{\mathrm{even}}$, the order complex $S_{C,\Gamma_G(C)}^{\mathrm{even}}$ is pure and shellable. This cannot be generalized to pseudographs.
    Actually, there are infinitely many pseudographs such that $S_{C,\Gamma_G(C)}^{\mathrm{odd}}$ is not shellable for some $C\in 2^{\cC_G}_{\mathrm{even}}$.
    \begin{figure}[h]
    \begin{center}
    \begin{subfigure}[h]{.15\textwidth}
    \centering
    \begin{tikzpicture}[scale=.7]
        \draw (1,-1) node{$G$};
    	\draw (0,0)--(2,0)--(2,2)--(0,2)--cycle;
    	\draw (0,2)..controls (1,2.7)..(2,2);
    	\fill (0,0) circle(2pt);
    	\fill (2,0) circle(2pt);
    	\fill (0,2) circle(2pt);
    	\fill (2,2) circle(2pt);
    	\draw (-0.3,0) node{$4$};
    	\draw (2.3,0) node{$3$};
    	\draw (-0.3,2) node{$1$};
    	\draw (2.3,2) node{$2$};
    	\draw (1,2.7) node{$a$};
    	\draw (1,1.7) node{$b$};
    \end{tikzpicture}
    \end{subfigure}
    \begin{subfigure}[h]{.3\textwidth}
    \centering
    \begin{tikzpicture}[scale=.7]
    	\node [draw] (1) at (0,0) {$1$};
    	\node [draw] (2) at (2,0) {$2$};
    	\node [draw] (3) at (4,0) {$3$};
    	\node [draw] (4) at (6,0) {$4$};
    	\node [draw] (13) at (4,1.7) {$13$};
    	\node [draw] (24) at (6,1.7) {$24$};
    	\node [draw] (12a) at (0,3.4) {$12a$};
    	\node [draw] (12b) at (2,3.4) {$12b$};
    	\node [draw] (134) at (4,3.4) {$134$};
    	\node [draw] (234) at (6,3.4) {$234$};
    	\node [draw] (123ab) at (0,5.1) {$123ab$};
    	\node [draw] (1234a) at (4,5.1) {$1234a$};
    	\node [draw] (1234b) at (6,5.1) {$1234b$};
    	\node [draw] (124ab) at (2,5.1) {$124ab$};
    	\path (1) edge (12a);
    	\path (1) edge (12b);
    	\path (1) edge (13);
    	\path (2) edge (24);
    	\path (2) edge (12a);
    	\path (2) edge (12b);
    	\path (3) edge (13);
    	\path (4) edge (24);
    	\path (24) edge (234);
    	\path (12a) edge (123ab);
    	\path (12a) edge (124ab);
    	\path (12b) edge (123ab);
    	\path (12b) edge (124ab);
    	\path (12a) edge (1234a);
    	\path (12b) edge (1234b);
    	\path (13) edge (134);
    	\draw[blue] (13) to [out=160,in=-60] (123ab);
    	\draw[blue] (24) to [out=160,in=-60] (124ab);
    	\path (134) edge (1234a);
    	\path (134) edge (1234b);
    	\path (234) edge (1234a);
    	\path (234) edge (1234b);
    \end{tikzpicture}
    \end{subfigure}
    \begin{subfigure}{.52\textwidth}
    \centering
    \begin{tikzpicture}[scale=.7]
    	\node [draw] (14) at (0,0) {$14$};
    	\node [draw] (23) at (1.8,0) {$23$};
    	\node [draw] (34) at (3.6,0) {$34$};
    	\node [draw] (12ab) at (-2.7,1.7) {$12ab$};
    	\node [draw] (124a) at (-.9,1.7) {$124a$};
    	\node [draw] (124b) at (.9,1.7) {$124b$};
    	\node [draw] (123a) at (2.7,1.7) {$123a$};
    	\node [draw] (123b) at (4.5,1.7) {$123b$};
    	\path (14) edge (124a);
    	\path (14) edge (124b);
    	\path (23) edge (123a);
    	\path (23) edge (123b);
    \end{tikzpicture}
    \end{subfigure}
    \end{center}
    \caption{Neither $S_{\cC_G,G}^{\mathrm{odd}}$ nor $S_{\cC_G,G}^{\mathrm{even}}$ is shellable.}\label{fig:non-shellable}
    \end{figure}
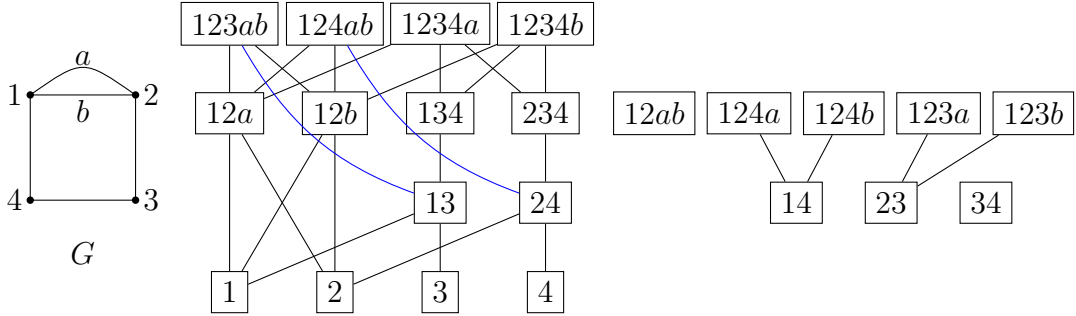

    Consider the pseudograph in Figure~\ref{fig:non-shellable}. Then, for $C=\cC_G$, $\Gamma_G(C)=G$ and neither $S_{C,G}^{\mathrm{odd}}$ nor $S_{C,G}^{\mathrm{even}}$ is shellable. However, both $S_{C,G}^{\mathrm{odd}}$ and $S_{C,G}^{\mathrm{even}}$ are homotopy equivalent to a wedge of spheres; $S_{C,G}^{\mathrm{odd}} \simeq S^2\vee S^2\vee S^2$  and $S_{C,G}^{\mathrm{even}}\simeq S^0\vee S^0 \vee S^0.$
    Therefore, as a further direction of research, it would be interesting to see whether $S_{C,\Gamma_G(C)}^{\mathrm{odd}}$ is homotopy equivalent to a wedge of spheres for any $C\in 2^{\cC_G}_{\mathrm{even}}$ or not. At this moment, all the examples we have are homotopy equivalent to a wedge of spheres. It would also be interesting to characterize pseudographs $G$ such that $S_{C,\Gamma_G(C)}^{\mathrm{even}}$ or $S_{C,\Gamma_G(C)}^{\mathrm{odd}}$ is shellable for any $C\in 2^{\cC_G}_{\mathrm{even}\ast}$.


\bigskip

\begin{thebibliography}{amsplain}

    \bibitem{BW1996}
    A.~Bj\"{o}rner and M.~L.~Wachs,
    \emph{Shellable nonpure complexes and posets. I.}, Trans. Amer. Math. Soc., \textbf{348} (4), 1996, 1299--1327.

    \bibitem{CD2006}
    M.~Carr and S.~L.~Devadoss,
    \emph{Coxeter complexes and graph-associahedra}, Topology Appl., \textbf{153} (12), 2006, 2155--2168.

    \bibitem{CDF2011}
    M.~Carr, S.~L.~Devadoss, S.~Forcey,
    \emph{Pseudograph associahedra}, J. Combin. Theory Ser. A, \textbf{118} (7), 2011, 2035–-2055.

    \bibitem{CKT}
    S.~Choi, S.~Kaji and S.~Theriault, \emph{Homotopy decomposition of a suspended real toric space}, to appear in the memorial volume for Professor Gitler in the Bol. Soc. Mat. Mex., arXiv:1503.07788.

    \bibitem{CP}
    S.~Choi and H.~Park, \emph{A new graph invariant arises in toric topology}, J. Math. Soc. Japan, \textbf{67}(2), 2015, 699--720.

    \bibitem{CP2}
    S.~Choi and H.~Park, \emph{On the cohomology and their torsion of real toric objects}, to appear in Forum Math., arXiv:1311.7056.

    \bibitem{Dan78}
    V.~I.~Danilov, \emph{The geometry of toric varieties}, Uspekhi Mat. Nauk \textbf{33}(2) (1978), 85–134.

    \bibitem{Jur80}
    J.~Jurkiewicz, \emph{Chow ring of projective nonsingular torus embedding}, Colloq. Math. \textbf{43}(2), 1980, 261–270.



    \bibitem{P05}
    A.~Postnikov, \emph{Permutohedra, associahedra, and beyond}, Int. Math. Res. Not. IMRN, \textbf{2009} (6), 2009, 1026--1106.

    \bibitem{ST2012}
    A. Suciu, A. Trevisan, \emph{Real toric varieties and abelian covers of generalized Davis–Januszkiewicz spaces}, preprint, 2012.

    \bibitem{Tre2012}
    A. Trevisan, \emph{Generalized Davis-Januszkiewicz spaces and their applications in algebra and topology}, Ph.D. thesis, Vrije University Amsterdam, 2012.

    \bibitem{Zel2006}
    A. Zelevinsky, \emph{Nested complexes and their polyhedral realizations}, Pure Appl. Math. Q. \textbf{2}(3), Special Issue: In honor of Robert D. MacPherson. Part 1, 2006, 655-–671.
\end{thebibliography}
\end{document}